\documentclass[11pt]{amsart}

\usepackage{hyperref}

\usepackage{graphicx}

\usepackage{xcolor}

\usepackage{amsmath, amssymb, amsfonts, amsthm, fullpage}
\usepackage{graphics,graphicx}
\usepackage{accents}
\usepackage{color}

\usepackage{algorithm}
\usepackage{algorithmicx}
\usepackage{algpseudocode}

\usepackage[margin=0.9in]{geometry}

\newcommand{\norma}[1]{{\left\vert\kern-0.25ex\left\vert\kern-0.25ex\left\vert #1
    \right\vert\kern-0.25ex\right\vert\kern-0.25ex\right\vert}}

\newcommand{\bu}{\mathbf{u}}

\newcommand{\bG}{\mathbf{G}}

\newcommand{\Div}{\nabla\!\cdot\!}

\newcommand{\tbn}[1]{{\left\vert\kern-0.25ex\left\vert\kern-0.25ex\left\vert #1 \right\vert\kern-0.25ex\right\vert\kern-0.25ex\right\vert}}

\newtheorem{remark}{Remark}[section]
\newtheorem{lemma}{Lemma}[section]

\newtheorem{theorem}{Theorem}[section]

\AtBeginDocument{%
  \setlength{\belowdisplayskip}{5pt} \setlength{\belowdisplayshortskip}{5pt}
  \setlength{\abovedisplayskip}{5pt} \setlength{\abovedisplayshortskip}{5pt}
}
\title{Oscillatory and regularized shock waves for a modified Serre-Green-Naghdi system}

\author{Daria Bolbot}
\address{\textbf{D.~Bolbot:} Computer, Electrical and Mathematical Science and Engineering Division King Abdullah University of Science and Technology (KAUST)
Thuwal 23955-6900, Saudi Arabia}
\email{bolbot.daria@gmail.com}

\author{Dimitrios Mitsotakis}
\address{\textbf{D.~Mitsotakis:} Victoria University of Wellington, School of Mathematics and Statistics, PO Box 600, Wellington 6140, New Zealand}
\email{dimitrios.mitsotakis@vuw.ac.nz}

\author{Athanasios E. Tzavaras}
\address{\textbf{A.E. Tzavaras} Computer, Electrical and Mathematical Science and Engineering Division King Abdullah University of Science and Technology (KAUST)
Thuwal 23955-6900, Saudi Arabia}
\email{athanasios.tzavaras@kaust.edu.sa}

\date{\today}

\dedicatory{In memory of Vassilios Dougalis whose work remains an inspiration}

\begin{document}

\begin{abstract}

The Serre-Green-Naghdi equations of water wave theory have been widely employed to study undular bores. In this study, we introduce a modified Serre-Green-Naghdi system incorporating the effect of an artificial term that results in dispersive and dissipative dynamics. We show that, over sufficiently extended time intervals, effectively approximates the classical Serre-Green-Naghdi equations and admits dispersive-diffusive shock waves as traveling wave solutions. 
The traveling waves converge to the entropic shock wave solution of the shallow water equations when the dispersion and diffusion approach zero in a moderate dispersion regime. These findings contribute to an understanding of the formation of dispersive shock waves in the classical Serre-Green-Naghdi equations and the effects of diffusion in the generation and propagation of undular bores.

\end{abstract}

\maketitle

\section{Introduction}

An undular bore is a wave disturbance occurring in a medium, such as the Earth's atmosphere or the free surface of the ocean, usually appearing in the form of a traveling front followed by oscillations (undulations). Our primary focus here will be on surface water waves, where undular bores commonly appear in estuaries, river mouths, and are generally categorized as shallow water waves. They play a crucial role in the design of both offshore and onshore marine structures. Perhaps the most renowned tidal bore is the one found in the Qiantang River in southeast China. This tidal bore is the world's largest, measuring approximately $4~m$ in height and spanning a width of $3~km$, as reported by \cite{Ch2012}.

Undular bores were initially studied through laboratory experiments by Favre, \cite{Favre1935}, and are often referred to as Favre's waves. Experimental evidence has shown that the shape of a tidal bore depends on its phase speed. The scaled phase speed of a wave is known as the Froude number ($Fr$), and its formula is given by $Fr=s/\sqrt{gD}$, where $s$ represents the wave phase speed, $g$ stands for the gravitational acceleration constant, and $D$ denotes the water depth (the distance from the undisturbed free water surface to the horizontal bottom). For small values of the Froude number, typically up to about $Fr = 1.3$, an undular bore consists of a traveling front followed by decreasing-magnitude oscillations or undulations, and turbulence-induced dissipation may not be significant. These undulations become smaller as the Froude number increases. When $Fr \geq 1.4$, these waves begin to break, causing the undulations to vanish and leading to dominant dissipative effects caused by turbulence \cite{Chanson2010,Ch2012}.

Numerical modeling of undular bores presents significant challenges due to the incomplete understanding of dissipation effects and the intricate nature of mathematical turbulence modeling. Perhaps the most accurate model for water waves capable of describing undular bores would be the Navier-Stokes equations \cite{KC2006,KC2009}. However, the Navier-Stokes equations have not yet been fully comprehended, and their numerical solutions are exceedingly complex. For these reasons, researchers investigating the generation and propagation of undular bores often turn to simplified mathematical models, such as various Boussinesq systems, which are designed to describe weakly non-linear and weakly dispersive water waves, \cite{CC2021}.

In his pioneering work  Peregrine \cite{Pere1966} introduced the first set of nonlinear and dispersive water wave equations specifically tailored for studying undular bores. As demonstrated in \cite{BMT2022}, Peregrine's system can provide a satisfactory description of undular bores, particularly for small values of the Froude number. Moreover, the incorporation of appropriate diffusive terms can effectively account for dissipative effects due to turbulence, thus extending the applicability of Peregrine's system to larger Froude numbers.

However, it is worth noting that the accuracy of predicted undular bore amplitudes using Peregrine's system has its limitations, as highlighted in \cite{BMT2022}. In the absence of dissipation, Peregrine's system is a weakly nonlinear and weakly dispersive model and such a model is ideal for small amplitude and long waves (compared to the depth). This property of Peregrine's system has prompted researchers to explore the extension of Boussinesq equations using alternative model equations that do not rely on the small amplitude assumption. One such model is the so-called Serre-Green-Naghdi equations that describe strongly nonlinear and weakly dispersive water waves. The Serre-Green-Naghdi equations were initially derived by Serre in \cite{Serre} to describe surface water waves over a constant bottom topography. Subsequently, Su and Gardner revisited these equations in their work \cite{SG1963}. The extension of Serre's ideas into two spatial dimensions, along with the consideration of variable bottom topography, was accomplished by Green and Naghdi in  \cite{GN1976}. 

Our aim is to elucidate the behavior of the Serre-Green-Naghdi (SGN) equations in describing undular bores, especially during their early stages of their formation. Specifically, we will study the effect on the SGN system of an artificial dissipation term, expressed in nondimensional and scaled variables as follows:
\begin{equation}\label{eq:Serre1}
\begin{aligned}
& h_t + (hu)_x = 0\ ,\\
& u_t + h_x + uu_x - \frac{\delta}{3h}\left[h^3(u_{xt}+uu_{xx}-u_x^2)\right]_x = \varepsilon \frac{(hu)_{xx}}{h}\ .
\end{aligned}
\end{equation}
Here, $x$ and $t$ represent (nondimensionalized) space and time variables, respectively. The quantity $h(x,t) = 1 + \eta(x,t)>0$ represents the total scaled depth, and $u(x,t)$ denotes the depth-averaged horizontal velocity of the fluid. The artificial term on the right-hand side of (\ref{eq:Serre1}) causes dissipation in the phase space of the traveling wave solutions and is employed to counterbalance the effects of dispersion during the formation of undular bores. The classical SGN equations are recovered when $\varepsilon=0$. The choice of the parameter $\varepsilon$ is based on experimentation
and, due to its heuristic nature, may serve as an approximation of bulk damping within the fluid's body. The SGN equations with wave breaking terms were studied in \cite{CC2021}, showing excellent agreement with undular bores, particularly those generated in an experimental dam-break setup. The SGN equations are known to be locally-well posed in time,  \cite{Israwi11}:  for  positive initial elevations  $h(x,0) \ge c_1 >0$, there is a $T>0$ and $c_0>0$ such that 
\begin{equation}\label{eq:cavity}
    h(x,t)\geq c_0>0\quad  \mbox{for $t\in[0,T]$}\ .
\end{equation}
This is the so-called {\em non-cavitation hypothesis} and we assume its validity here for all $\varepsilon \in[0,1)$, at least locally in time.

% It is easy to see that solutions of (\ref{eq:Serre1}) satisfy the mass and momentum conservation laws of the SGN but not the total energy conservation. The energy of the SGN system is dissipated by the solutions of (\ref{eq:Serre1}) for both linear waves and weakly nonlinear waves. On the other hand, as we will demonstrate in this paper, the specific artificial term is suitable for explaining the classical SGN equations' ability to describe undular bores over appropriately large time scales. 
% Even though this term is heuristically chosen, we will demonstrate that it is sufficient to accurately describe the behavior of even breaking undular bores. 

The reasons for choosing a model with the particular artificial term to study undular bores are the following: (i) There are no dissipative fully nonlinear wave equations derived in the literature thus far, \cite{DH2021}; (ii) Nonlinear and dispersive wave equations in the presence of dissipative dynamics exhibit special traveling wave solutions that resemble undular bores  \cite{J1972,BS1985,BMT2022}; (iii) The approximation properties of the modified system in relation to the classical SGN equations justify the formation of dispersive shock waves in the latter from initial conditions that connect different states \cite{KKM2016, ZPR2017}. 

Traveling wave solutions appearing in dispersive-diffusive dynamical systems and connect different states are called dissipative-dispersive shock waves \cite{EHS2017}. When the dissipation is sufficiently small compared to the dispersion, then the traveling waves consist of a traveling front followed by undulations that we will call  oscillatory shock waves \cite{HN2007}. When the dissipation exceeds a certain limiting value, then the solution becomes monotonic and loses its oscillatory nature. We will refer to such traveling waves as regularized shock waves since they have a form closer to classical entropic shock waves, and even converge to such solutions when the dispersion and the dissipation tend to zero at an appropriate rate, \cite{BMT2023}.

In the present work we carry two objectives:  (i)  We investigate the existence and characteristics of traveling wave solutions for the modified SGN equations (\ref{eq:Serre1}), building upon the groundwork laid by \cite{BS1985}. (ii) We show that solutions of the modified SGN equations with $\varepsilon>0$, are effective approximations of the corresponding solutions of the classical SGN equations with $\varepsilon=0$ emanating from the same initial conditions. Our analysis reveals that the discrepancy between the two  solutions is of order $O(\varepsilon t)$. This implies that these solutions remain in close proximity for time intervals of the order of $O(1/\varepsilon)$. For larger time intervals, Whitham's modulation theory has been employed to describe the evolution of dispersive shock waves, which also resemble in some situations undular bores \cite{EH2016,EGS2006,EGS2008}. These  findings complement the numerical results of \cite{PZR2018} about the formation of dispersive shock waves, demonstrating that it is rather impossible for regularized shock waves to appear in the absence of dissipation in the SGN equations.

The structure of the paper is as follows: The dynamics of the traveling wave solutions of the modified SGN equations is studied in Section \ref{sec:dynamics}. The proof of the existence of such solutions is presented in Section \ref{sec:existence}. The dichotomy of the dissipative-dispersive shock waves in oscillatory and regularized is presented in Section \ref{sec:shapes}. The limiting behavior of the modified SGN system is studied in Section \ref{sec:comparison}. We justify that solutions of the modified and classical SGN equations remain close for small dissipations at a time scale $O(\varepsilon t)$ and converge when the dissipation is vanishing. Moreover, we review the convergence theory of the traveling wave solutions to entropic shock waves of the shallow water equations when the  dispersion and diffusion tend to zero simultaneously in the moderate dispersion regime. In Section \ref{sec:comparison}  we also present a numerical investigation assessing the precision of the mathematical model.

\section{The dynamics of the traveling waves}\label{sec:dynamics}

The main objective of this work is to study traveling wave solutions of the system (\ref{eq:Serre1}) that resemble undular bores. Before establishing the existence of such traveling waves we study their dynamical picture. For convenience, we write the system (\ref{eq:Serre1}) in conservative variables as
\begin{equation}
\label{eq:Serre2}
\begin{aligned}
& h_t+q_x=0\ ,   \\
& q_t+\left(\frac{1}{2}h^2+\frac{q^2}{h}\right)_x-\frac{\delta}{3}\gamma_x=\varepsilon q_{xx}\ ,
\end{aligned}
\end{equation}
where here $q=hu$ is the so called momentum, and $\gamma=h^3[u_{xt}+uu_{xx}-u_x^2]$ is an approximation of the fluid vertical acceleration at the bottom. The non-dispersive shallow water waves equations coincides with the hyperbolic part of the SGN equations (\ref{eq:Serre2}) for $\delta=\varepsilon=0$
\begin{equation}
\label{SW}
\begin{aligned}
& h_t+q_x=0\ ,   \\
& q_t+\left(\frac{1}{2}h^2+\frac{q^2}{h}\right)_x=0\ .\
\end{aligned}
\end{equation}

For the description of the propagation of a traveling wave with speed $s>0$ we consider the Riemann initial data
\begin{equation}
(h(x,0), q(x,0))=(h_0, q_0) 
=
    \begin{cases}
& (h_r, q_r), \text{ for } x\leq 0\ ,   \\
& (h_l, q_l), \text{ for } x>0 \ .
\end{cases}
\end{equation}
For the particular problem, the Jacobian matrix is the matrix
\begin{equation}
    \begin{pmatrix}  0 & 1\\ h-\frac{q^2}{h^2} & 2\frac{q}{h} \end{pmatrix}\ ,
\end{equation}
and has eigenvalues
\begin{equation}
\lambda_{1,2}=\frac{q}{h} \pm \sqrt{h}\ .
\end{equation}
The Rankine-Hugoniot shock conditions for this system are 
\begin{equation}
\label{RH}
s= \frac{q_l-q_r}{h_l-h_r}\quad \text{ and }\quad s= \frac{\frac{h_l^2}{2}-\frac{h_r^2}{2}+\frac{q_l^2}{h_l}-\frac{q_r^2}{h_r}}{q_l-q_r}\ .
\end{equation}
Without loss of generality, we consider the phase speed $s>0$ and assume the admissibility Lax condition 
$$\lambda_2(U_r)<s<\lambda_2(U_l)\ ,$$
or equivalently
\begin{equation}
\label{ADM}
\frac{q_r}{h_r}+\sqrt{h_r}<s<\frac{q_l}{h_l}+\sqrt{h_l}\ .
\end{equation}

Solving equations (\ref{RH}), given $s$, $h_r$ and $q_r$ and assuming that $h_l>h_r>0$ we obtain the following formulas
\begin{equation}\label{eq:leftlims}
h_l=\frac{-h_r + \sqrt{h_r^2 + 8h_r( q_r/h_r- s)^2}}{2}, \quad q_l=\frac{ 2 q_r - 3 s h_r  +s \sqrt{h_r^2 + 8 h_r (q_r/h_r -  s)^2}}{2}\ .
\end{equation}
The fact that $h_l>h_r$ follows from the admissibility conditions (\ref{ADM}).

The intrinsic properties of the SGN system of equations are linked to the behavior of water waves. Since oscillatory shock waves are dissipative-dispersive perturbations of entropic shocks of the hyperbolic shallow water wave equations, it is expected that the Lax conditions will influence the dissipative-dispersive shock waves of the modified SGN equations. For these reasons, we assume the validity of Lax conditions (\ref{ADM}) for the traveling wave solutions of the modified SGN equations (\ref{eq:Serre2}).

Let $U_r=(h_r,q_r)$ be the right limit of the traveling wave solution of (\ref{eq:Serre2}). We will determine all the possible states $U=(h,q)$ that can be connected to $U_r$ by a front shock.

% Combining the Rankine-Hugoniot conditions (\ref{RH}) we obtain
% \begin{equation}
% \label{RS}
% (q-q_r)^2=(h-h_r)\left(\frac{1}{2}\left(h^2-h_r^2\right)+\frac{q^2}{h}-\frac{q_r^2}{h_r}\right)\ .
% \end{equation}
% From the relation \eqref{ADM} we have that the second parenthesis on the right-hand side of \eqref{RS} is positive, thus $h-h_r>0$ and, from the first condition of \eqref{RH}, we deduce that $q>q_r$, and thus we write
% \begin{equation}
% \label{RSCORRECT}
% q-q_r=\sqrt{(h-h_r)\left[\left(\frac{1}{2}h^2+\frac{q^2}{h}\right)-\left(\frac{1}{2}h_r^2+\frac{q_r^2}{h_r}\right)\right]}\ .
% \end{equation}
%By choosing one of the parameters $s$, $q_l$ and $h_l$, then the other two can be obtained from the Rankine-Hugoniot condition and from \eqref{RSCORRECT}.

Consider now the modified SGN equations written in the form
\begin{equation}
\label{Serre}
\begin{aligned}
& h_t+q_x=0\ ,   \\
& q_t+\left(\frac{1}{2}h^2+\frac{q^2}{h}\right)_x-\frac{\delta}{3}\gamma_x=\varepsilon q_{xx}\ ,
\end{aligned}
\end{equation}
where $\gamma=h^3(u_{tx}+uu_{xx}-u_{x}^2)$. Since we are looking for travelling wave solutions, we consider the {\em ansatz} $$h(x,t)=\zeta(\xi) \quad \text{ and } \quad q(x,t)=w(\xi)\ ,$$
where 
$$\xi=-\frac{1}{\sqrt{\delta}}(x-st)\ .$$
Recall that the variable $\zeta>0$ denotes the total depth (distance between the bottom and the free-surface of the water).

\begin{remark}\label{rem:lcremark}
If $w_l=0$ and $\zeta_l=1$, then for $s>0$  we have $K_1=-s<0$ without the Lax condition (\ref{eq:lc}). On the other hand, for general traveling waves, the Lax condition apparently must be imposed.
\end{remark}

We assume that 
\begin{equation}\label{eq:limits}
\lim_{\xi\to+\infty}\zeta(\xi)=\zeta_r, \quad \lim_{\xi\to-\infty} \zeta(\xi)=\zeta_l,\quad \lim_{\xi\to+\infty} w(\xi)=w_r,\quad \text{and}\quad \lim_{\xi\to-\infty}w(\xi)=w_l\ , 
\end{equation}
with $\zeta_l=h_r<h_l=\zeta_r$, and $w_l=q_r<q_l=w_r$. We also assume that for $i=1,2,\dots$ all the derivatives of $\zeta$ and $w$ vanish at infinity:
$$\lim_{|\xi|\to \pm\infty} \zeta^{(i)}=\lim_{|\xi|\to \pm\infty} w^{(i)}=0\ .$$

Note that the Rankine-Hugoniot conditions become
\begin{equation}\label{eq:rh2}
s= \frac{w_l-w_r}{\zeta_l-\zeta_r}\quad \text{ and }\quad s= \frac{\frac{\zeta_l^2}{2}-\frac{\zeta_r^2}{2}+\frac{w_l^2}{\zeta_l}-\frac{w_r^2}{\zeta_r}}{w_l-w_r}\ ,
\end{equation}
and the Lax condition
\begin{equation}
\label{eq:lc}
\frac{w_l}{\zeta_l}+\sqrt{\zeta_l}<s<\frac{w_r}{\zeta_r}+\sqrt{\zeta_r}\ .
\end{equation}
Moreover, relationships (\ref{eq:limits}) take the form
\begin{equation}\label{eq:leftlims2}
\zeta_r=\frac{-\zeta_l + \sqrt{\zeta_l^2 + 8\zeta_l( w_l/\zeta_l- s)^2}}{2}, \quad w_r=\frac{ 2 w_l - 3 s \zeta_l  +s \sqrt{\zeta_l^2 + 8 \zeta_l (w_l/\zeta_l -  s)^2}}{2}\ .
\end{equation}

This change of variables leads to the system 
\begin{equation}
\label{SerreTW}
\begin{aligned}
& -s\zeta'+w'=0\ ,   \\
& -sw'+\left(\frac{\zeta^2}{2}+\frac{w^2}{\zeta}\right)'-\frac{1}{3}\frac{d}{d\xi}\left\{\zeta^3\left(\frac{w}{\zeta}\right)''\left(\frac{w}{\zeta}-s\right)-\left[\left(\frac{w}{\zeta}\right)'\right]^2\right\}=-\frac{\varepsilon}{\sqrt{\delta}} w''\ ,\
\end{aligned}
\end{equation}
where the $'=\tfrac{d}{d\xi}$ denotes the ordinary derivative with respect to $\xi$.

Integrating system (\ref{SerreTW}) from $\xi$ to $\infty$ yields
\begin{equation}
\label{SerreTWInt}
\begin{aligned}
& -s\zeta+w=K_1\ ,   \\
& -sw+\frac{\zeta^2}{2}+\frac{w^2}{\zeta}-\frac{\zeta^3}{3}\left\{\left(\frac{w}{\zeta}\right)''\left(\frac{w}{\zeta}-s\right)-\left[\left(\frac{w}{\zeta}\right)'\right]^2\right\}=-\frac{\varepsilon}{\sqrt{\delta}} w'+K_2\ ,
\end{aligned}
\end{equation}
where 
\begin{equation}\label{eq:defk1}
K_1=-s\zeta_l+w_l=-s\zeta_r+w_r\ ,
\end{equation} and
\begin{equation}\label{eq:defk2}
K_2=-sw_l+\frac{\zeta_l^2}{2}+\frac{w_l^2}{\zeta_l}=-sw_r+\frac{\zeta_r^2}{2}+\frac{w_r^2}{\zeta_r}\ ,
\end{equation}
are constants. Note that 
\begin{equation}\label{eq:ineK1}
    K_1<0\ ,
\end{equation} since $-s\zeta_l+w_l<-\zeta_l\sqrt{\zeta_l}<0$ due to (\ref{eq:lc}). The last inequality gives the lowest bound for the speed $s$ to be
\begin{equation}\label{eq:speed}
    s>\frac{w_l+\zeta_l\sqrt{\zeta_l}}{\zeta_l}\ .
\end{equation}
Moreover, we have
\begin{equation}
    K_2-s K_1=\frac{K_1^2}{\zeta_l}+\frac{\zeta_l^2}{2}=\frac{K_1^2}{\zeta_r}+\frac{\zeta_r^2}{2}>0\ .
\end{equation}
Eliminating  $w$ in the system (\ref{SerreTWInt}) and dividing by $\zeta>0$ we obtain
\begin{equation}
\label{ODE2}
\frac{\zeta}{2}+\left(\frac{K_1}{\zeta}\right)^2-\frac{K_2-sK_1}{\zeta}+\frac{K_1^2}{3}\left(\frac{\zeta'}{\zeta}\right)'=-c\frac{\zeta'}{\zeta}\ ,
\end{equation}
where $c=s\varepsilon/\sqrt{\delta}$.

Define $z=\ln{\zeta}$. Then, we can write the equation (\ref{ODE2}) in a more compact form
\begin{equation}
\label{ODE2z}
\frac{K_1^2}{3}z''+F(e^z)=-cz'\ ,
\end{equation}
where 
\begin{equation}\label{eq:functionF}
F(e^z)=\frac{\zeta}{2}+\left(\frac{K_1}{\zeta}\right)^2-\frac{K_2-sK_1}{\zeta}=F(\zeta)\ .
\end{equation}
It is easy to see that $F(\zeta_l)=0$. Performing long division we rewrite $F$ in the factorized form
\begin{equation}\label{eq:ffact}
F(\zeta)=\frac{(\zeta-\zeta_l)\left(\zeta^2+\zeta\zeta_l-2\frac{K_1^2}{\zeta_l}\right)}{2\zeta^2}\ .
\end{equation}

\begin{lemma}\label{lem:limits}
Suppose that the equation (\ref{ODE2}) has a solution that satisfies conditions (\ref{eq:limits}) and (\ref{eq:lc}). Then we have
\begin{equation}\label{eq:spboundrelz}
\zeta_r= \frac{K_1 - w_l}{2s} +\frac{1}{2}\sqrt{\frac{8K_1^2+\zeta_l^3}{\zeta_l}}= -\frac{\zeta_l}{2} +\frac{1}{2}\sqrt{\frac{8(-s \zeta_l+  w_l)^2+\zeta_l^3}{\zeta_l}}>\zeta_l\ ,
\end{equation}
and
\begin{equation}\label{eq:spboundrel}
w_r = \frac{3K_1-w_l}{2}+\frac{s\sqrt{8K_1^2+\zeta_l^3}}{2\sqrt{\zeta_l}}=\frac{2w_l -3 s \zeta_l}{2} + \frac{s \sqrt{8(-s \zeta_l+  w_l)^2+\zeta_l^3}}{2 \sqrt{\zeta_l}}>w_l \ .
\end{equation}
\end{lemma}
\begin{proof}
Taking the limit $\xi\to+\infty$ in (\ref{ODE2z}) leads to the equation
$$ 
F(e^{z_r})=F(\zeta_r)=0\ ,
$$
which has three roots, namely, 
$$
\begin{aligned}
\zeta_r&=\zeta_l\ ,\\
\zeta_r&=-\frac{\zeta_l}{2} \pm\frac{1}{2}\sqrt{\frac{8K_1^2+\zeta_l^3}{\zeta_l}}\ .
\end{aligned}
$$
Because we have assumed that $\zeta>0$ we exclude the negative root. Now, let $\zeta_r=\zeta_l$. We multiply (\ref{ODE2z}) with $z$ and integrate over $\mathbb{R}$. Because $F(\zeta)$ is bounded for $\zeta>0$, we have that 
$$
-c\int_{-\infty}^{\infty}(z')^2~d\xi=\int_{-\infty}^{+\infty} F(e^z)z'~d\xi =\int_{z_l}^{z_r} F(e^z)~dz=0\ .
$$
This means that $z'=0$ for all $\xi\in\mathbb{R}$ and thus the solution is constant. 
This contradicts with the assumption that the solution is not constant, and therefore
we have to exclude the case $\zeta_r=\zeta_l$. The only remaining possibility is the case (\ref{eq:spboundrelz}). Using the first equation of (\ref{SerreTWInt}) we obtain also the equation (\ref{eq:spboundrel}).

Finally, we verify that $w_r>w_l$ and $\zeta_r>\zeta_l$: From the Lax condition (\ref{eq:lc}) we have that $s\zeta_l-w_l>\zeta_l\sqrt{\zeta_l}$, which implies
\begin{equation}
K_1^2>\zeta_l^3\ .
\end{equation}
Therefore, we have that
$$
\zeta_r =-\frac{\zeta_l}{2}+\frac{1}{2}\sqrt{\frac{8K_1^2+\zeta_l^3}{\zeta_l}}>-\frac{\zeta_l}{2}+\frac{3\zeta_l}{2}=\zeta_l\ .
$$
Similarly, we verify that
$$
w_r = \frac{3K_1-w_l}{2}+\frac{s\sqrt{8K_1^2+\zeta_l^3}}{2\sqrt{\zeta_l}}> \frac{3K_1-w_l}{2}+\frac{s\sqrt{9\zeta_l^3}}{2\sqrt{\zeta_l}}=\frac{2w_l-3s\zeta_l}{2}+\frac{3s\zeta_l}{2}=w_l\ ,
$$
which completes the proof.
\end{proof}

\begin{remark}
In the special case where $w_l=0$ and $\zeta_l=1$ the equation (\ref{eq:spboundrel}) is simplified to 
\begin{equation}\label{eq:spbounrel2}
w_r=s\frac{-3 + \sqrt{1 + 8 s^2}}{2}>0\ ,
\end{equation}
which for $s>1$ implies $w_r>0$. Similarly, the equation (\ref{eq:spboundrelz}) is simplified to
\begin{equation}\label{eq:spbounrel2z}
\zeta_r=\frac{-1 + \sqrt{1 + 8 s^2}}{2}>1\ .
\end{equation}
\end{remark}

It is worth mentioning that the function $F$ can be factorized into
\begin{equation}\label{eq:ffact2}
F(\zeta)=\frac{(\zeta-\zeta_l)(\zeta-\zeta_r)(\zeta+\zeta_l+\zeta_r)}{2\zeta^2}\ .
\end{equation}
Introducing the auxiliary variable $v=z'$ we write the equation (\ref{ODE2z}) as a system of first order equations:
\begin{equation}
\label{DynSys}
\begin{aligned}
& z'=v\ ,  \\
& v'=-\frac{3}{K_1^2} F(e^z)-\frac{3c}{K_1^2} v\ .
\end{aligned}
\end{equation}
For the sake of convenience, we will use the notation $\bu'=\bG(\bu)$ for the system (\ref{DynSys}) where 
$$\bu=\begin{pmatrix}z\\v\end{pmatrix}\quad \text{and}\quad 
\bG(\bu)=\begin{pmatrix}
    v\\
    -\frac{3}{K_1^2}F(e^z)-\frac{3c}{K_1^2}v
    \end{pmatrix}\ .
$$

For the analysis of the dynamics of this particular system we compute its critical points. Obviously $v=0$ is a requirement for a critical point because of the first equation of (\ref{DynSys}). From (\ref{eq:ffact}) we obtain three zeros for the function $F(\zeta)$. In particular, we have the zeros $\zeta_1=\zeta_l$, $\zeta_2=\zeta_r$ and $\zeta_3=-(\zeta_l+\zeta_r)$. 
we obtain the zeros
$$
\begin{aligned}
& (\zeta_1,v_1)=(\zeta_l,0)\ ,\\
& (\zeta_2,v_2)=(\zeta_r,0)=\left(-\frac{\zeta_l}{2}+\frac{1}{2}\sqrt{\frac{8K_1^2+\zeta_l^3}{\zeta_l}},0\right)\ ,\\
& (\zeta_3,v_3)=(-(\zeta_l+\zeta_r),0)=\left(-\frac{\zeta_l}{2}-\frac{1}{2}\sqrt{\frac{8K_1^2+\zeta_l^3}{\zeta_l}},0\right)\ .
\end{aligned}
$$ 
The zeros of $F$, in $(z,v)$ variables, are $(\ln{\zeta_l},0)$ and $(\ln{\zeta_r},0)$.
However, because $\zeta>0$ is the total depth, we have to exclude $\zeta_3$.

To study the stability properties of the critical points we linearize system (\ref{DynSys}) around the point $(z_i,v_i)^T=(z_i,0)^T$ for $i=1,2$. We can write the corresponding linearization as
\begin{equation}
\label{DynSysLin}
\begin{pmatrix}  z-z_i\\ v \end{pmatrix}'
    =
\begin{pmatrix}  0 & 1
 \\ -\frac{3}{K_1^2}F'(e^{z_i})e^{z_i} & -\frac{3c}{K_1^2} \end{pmatrix} \begin{pmatrix}  z-z_i\\ v \end{pmatrix}\ ,
\end{equation}
where
\begin{equation}\label{eq:fderiv}
F'(\zeta)=\left(\frac{\zeta}{2}+\frac{K_1^2}{\zeta^2}-\frac{K_2-sK_1}{\zeta}\right)'
=\frac{\zeta^3+2(K_2-sK_1)\zeta-4K_1^2}{2\zeta^3}
=\frac{\zeta^3+\frac{2K_1^2+\zeta_l^3}{\zeta_l}\zeta-4K_1^2}{2\zeta^3}\ .
\end{equation}

The characteristic polynomial of the Jacobian matrix $\bG'(\bu)$ is 
\begin{equation}\label{eq:characteristiceq}
\lambda^2+\frac{3c}{K_1^2}\lambda+\frac{3}{K_1^2}F'(e^{z_i})e^{z_i}=0\ ,
\end{equation}
thus the eigenvalues of the Jacobian matrix are 
\begin{equation}\label{eq:eigs}
\lambda^i_{\pm}=-\frac{3c}{2K_1^2}\left[1\pm\sqrt{1-\frac{4K_1^2}{3c^2}F'(e^{z_i})e^{z_i}}\right],\qquad \text{for $i=r, l$} \ .
\end{equation}

From the admissibility condition (\ref{eq:lc}), we have $(s-\frac{w_l}{\zeta_l})^2>\zeta_l$ and $(s-\frac{w_r}{\zeta_r})^2<\zeta_r$. Writing 
$$K_1^2=\zeta_l^2\left(s-\frac{w_l}{\zeta_l}\right)^2=\zeta_r^2\left(s-\frac{w_r}{\zeta_r}\right)^2\ ,$$ 
yields $\zeta_l^3<K_1^2<\zeta_r^3$. Thus, from the relation (\ref{eq:fderiv}) and for $z_i=\zeta_l$ we have that
\begin{equation}\label{eq:concav}
F'(\zeta_l)=\frac{\zeta_l^3-K_1^2}{\zeta_l^3}<0\ .
\end{equation}
On the other hand, for $z_i=\zeta_r$ we have
\begin{equation}\label{eq:convex}
F'(\zeta_r)=\frac{\zeta_r^3-K_1^2}{\zeta_r^3}>0\ .
\end{equation}
Therefore, $\zeta_l$ is a saddle point (eigenvalues are real, distinct, and of opposite signs), and the point $\zeta_r$ is either stable spiral or nodal point. In particular,
$$
\zeta_r = \begin{cases}
  \text{stable spiral}, & \text{if }c^2< \frac{4}{3}K_1^2 F'(\zeta_r)\zeta_r \\
  \text{stable node}, & \text{if } c^2\geq \frac{4}{3}K_1^2 F'(\zeta_r)\zeta_r\\
\end{cases}\ .
$$

In the cases where $\zeta_r$ is a spiral, the solution will be an oscillatory shock wave while when $\zeta_r$ is a node then the solution will be a regularized shock wave \cite{BMT2022}. The existence of such orbits is established in the next section.

\begin{remark}
So far, we have assumed that $\zeta_l < \zeta_r$ and $w_l < w_r$. Showing the existence of traveling waves with such properties is sufficient to demonstrate that traveling waves with $\zeta_l > \zeta_r$ and $w_l > w_r$ also exist. This is due to the symmetry: 
$$\xi \to -\xi, \quad s \to -s, \quad w \to -w, \quad \text{and} \quad \zeta \to \zeta.$$
Furthermore, it is not possible to have $\zeta_l < \zeta_r$ and $q_l > q_r$ with $s > 0$ (front shock), because in such a case, according to the definition of $K_1$ in equation (\ref{eq:defk1}), we would have 
$s=(q_l - q_r)/(\zeta_l - \zeta_r) < 0$.
\end{remark}

The existence of traveling waves for a system in the form  (\ref{DynSys}) has been shown in \cite[Theorem 1]{BMT2022}. In order to study the solutions to the system (\ref{DynSys}) we will follow the steps of the work \cite{BS1985}. This methodology revealed a new pattern for the fast traveling dispersive shocks of the bidirectional Peregrine system \cite{BMT2022}, which has certain differences compared to the traveling waves of the unidirectional KdV-Burgers analog \cite{BS1985}). For the case of the SGN equations though we will show that the traveling wave solutions are similar to those of the KdV equation of \cite{BS1985}. Under the assumption of existence of solutions, we show that the solution $\zeta$ remains always above the steady state $\zeta_l$:
\begin{lemma}\label{lem:bounds}
If $(\zeta,w)$ is a unique non-constant solution to (\ref{ODE2}) that satisfies the conditions (\ref{eq:limits})--(\ref{eq:lc}) and (\ref{eq:speed}), then 
$$\zeta_l<\zeta(x)<\left(\frac{K_1}{\zeta_l}\right)^2\ .$$
\end{lemma}
\begin{proof}
We first multiply equation (\ref{ODE2z}) with $z'$ and we integrate over $(-\infty,\xi)$ for some $\xi\in\mathbb{R}$. This leads to
\begin{equation}\label{eq:ineqif}
\int_{-\infty}^\xi F(e^z)z'~dy=-\frac{K_1^2}{6}(z'(\xi))^2-c\int_{-\infty}^y (z'(y))^2~dy<0\ .
\end{equation}
On the other hand, using (\ref{eq:functionF}) we get
\begin{equation}\label{eq:ineqif2}
\int_{-\infty}^\xi F(e^z)z'~dy=\frac{(\zeta_l-\zeta(\xi))^2[\zeta_l^2\zeta(\xi)-K_1^2]}{2\zeta_l^2\zeta^2(\xi)}<0\ .
\end{equation}

This implies that $\zeta(\xi)\not=\zeta_l$ for all $\xi\in\mathbb{R}$, and since $\zeta(\xi)\to\zeta_r>\zeta_l$ as $\xi\to\infty$ we have that $\zeta(\xi)>\zeta_l$. Moreover, from the inequality (\ref{eq:ineqif2}) we get that $\zeta(\xi)<K_1^2/\zeta_l^2$ for all $\xi\in \mathbb{R}$, and this completes the proof.
\end{proof}

The result of this lemma is aligned with the pattern of commonly observed undular bores. Next, we estimate the local extrema of the solution.

\begin{lemma}\label{lem:locextr}
Let $(\zeta,w)$ be a unique non-constant solution to (\ref{ODE2}) that satisfies the conditions (\ref{eq:limits})--(\ref{eq:lc}) and (\ref{eq:speed}). If $x_0\in\mathbb{R}$ is such that $\zeta'(x_0)=0$, then $x_0$ is an isolated extremum. Moreover,  $\zeta(x_0)$ is a local maximum if $\zeta(x_0)>\zeta_r$, while $\zeta(x_0)$ is a local minimum if $\zeta_l<\zeta(x_0)<\zeta_r$.
\end{lemma}
\begin{proof}
Note that if $x_0$ is such that $\zeta'(x_0)=0$, then $z'(x_0)=0$ as well.  Moreover, we have that $z''(x_0)=\zeta''(x_0)e^{\zeta(x_0)}$. If $z''(x_0)=0$, then by (\ref{ODE2z}) we have $F(\zeta(x_0))=0$, which because of Lemma \ref{lem:bounds} can only be true if $\zeta(x_0)=\zeta_r$. Also in that case, we can see that after differentiation of (\ref{ODE2z}) all the derivatives of $z$ at $x_0$ are zero, and thus $z$ will be constant, which contradicts the hypothesis. Thus, $x_0$ should be an isolated extremum. 

If $x_0$ is a local maximum, then $z''(x_0)<0$, which reduces the equation (\ref{ODE2z}) to the inequality $F(\zeta)>0$. This implies that 
$$(\zeta(x_0)-\zeta_l)\left(\zeta^2(x_0)-\zeta(x_0)\zeta_l-2\frac{K_1^2}{\zeta_l}\right)>0\ ,$$
which means that $\zeta^2(x_0)-\zeta(x_0)\zeta_l-2\frac{K_1^2}{\zeta_l}>0$ since Lemma \ref{lem:bounds} excludes the possibility $\zeta(x_0)<\zeta_l$. Solving the quadratic inequality we obtain $\zeta(x_0)>\zeta_r$. If $\zeta(x_0)$ is a local minimum, then we obtain similarly that $\zeta_l<\zeta(x_0)<\zeta_r$.
\end{proof}

We turn now our attention to the energy of the system (\ref{DynSys}). Specifically, we define the energy of (\ref{DynSys}) via the Liapunov function 
\begin{equation}
\label{Hamqhd}
\begin{aligned}
H(z,v)=\frac{v^2}{2}+\Phi(z)\ ,
\end{aligned}    
\end{equation}
with $\Phi'(z)=\frac{3}{K_1^2}F(e^z)$. The potential (normalized with respect to $z_l$) can be written as
\begin{equation}\label{eq:potential}
    \Phi(z)=\frac{3}{K_1^2}\int_{z_l}^zF(e^z)~dz=-\frac{3}{K_1^2}\frac{(e^z - e^{z_l})^2 (K_1^2 - e^{z} e^{2z_l})}{2 e^{2z}e^{2z_l}}=-\frac{3}{K_1^2}\frac{(\zeta - \zeta_l)^2 (K_1^2 - \zeta \zeta_l^2)}{2 \zeta^2 \zeta_l^2}\ .
\end{equation} 
  % \frac{3 e^{-2z}(e^{3z}-K_1^2+2(K_2-sK_1)e^z)}{2K_1^2}+C\ ,$$
Figure \ref{fig:potential} depicts a typical graph of the function $\Phi$. The shape of $\Phi$ can be determined easily by studying its intervals of monotonicity using formula (\ref{eq:ffact2}) and its convexity using the formulas (\ref{eq:fderiv}), (\ref{eq:concav}) and (\ref{eq:convex}). For the specific figure we used $\zeta_l=1$, $w_l=0$ and $s=2$.

\begin{figure}[ht!]
  \centering
\includegraphics[width=\columnwidth]{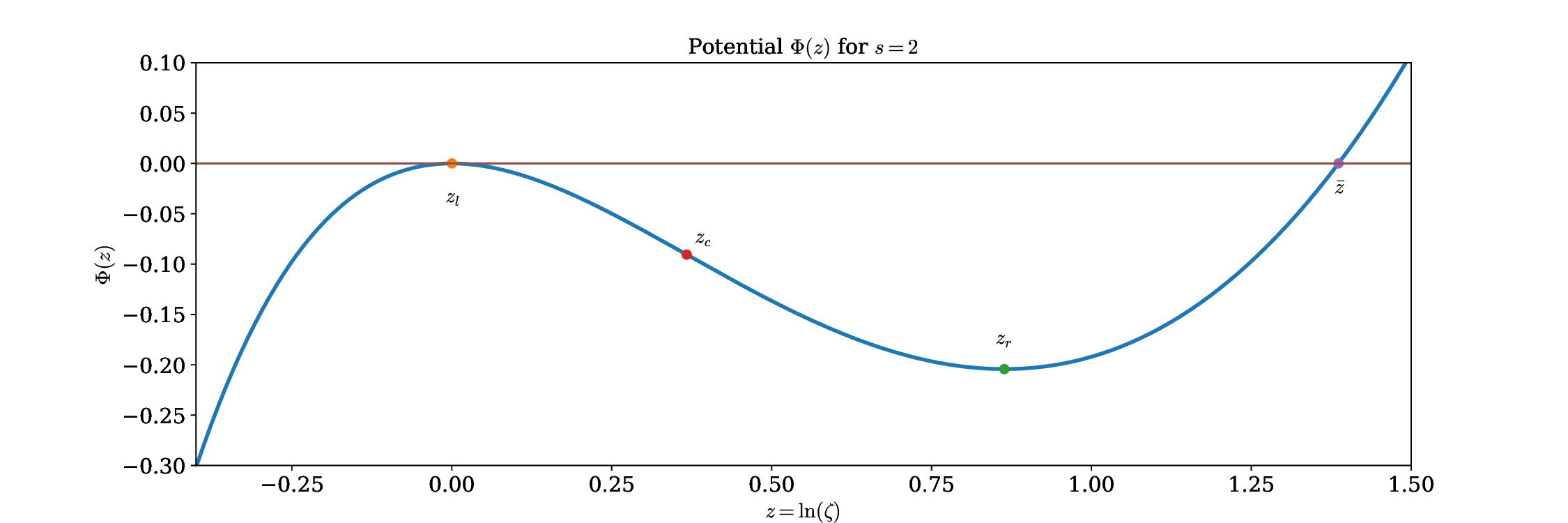}
  \caption{A typical potential function $\Phi(z)$}
  \label{fig:potential}
\end{figure}

With this definition for the energy, we conclude that the system (\ref{DynSys}) is dissipative in the sense that the Liapunov function $V(\xi)=H(z(\xi),v(\xi))$ is decreasing:
\begin{equation}\label{eq:dissipation}
\frac{d}{d\xi}V(\xi)=\frac{d}{d\xi}H(z,v)=-\frac{3c}{K_1^2} v^2<0\ .
\end{equation}
Thus, for $\xi\in\mathbb{R}$ we have that $V(\xi)<\lim_{\xi\to-\infty} V(\xi)$. This means that $H(z,0)<H(z_l,v_l)=H(z_l,0)$. This leads to the inequality
$$-\frac{(\zeta-\zeta_l)^2(K_1^2-\zeta\zeta_l^2)}{\zeta^2\zeta_l^2}<0\ .$$ Since $\zeta>\zeta_l$, we have that 
\begin{equation}
\zeta(\xi)<\bar{\zeta}=\left(\frac{K_1}{\zeta_l}\right)^2\ .
\end{equation}
(We will denote $\bar{z}=\ln\bar{\zeta}$).
Of course this was known from Lemma \ref{lem:bounds}. Thus, contrary to the Boussinesq system of \cite{BMT2023}, the information we can extract from the Liapunov function in terms of the behaviour of the maxima of the solution is the same as in the case of the KdV equation. Also asymptotically the solutions of the SGN equations do not concentrate at a maximum value $\bar{\zeta}=e^{\bar{z}}$. Moreover, the dissipation property (\ref{eq:dissipation}) of the Liapunov function $V$ is the reason for calling the new term in the SGN equations a dissipative term.

The potential function is defined for all $z\in \mathbb{R}$ but because $z=\ln\zeta$ with $\zeta>\zeta_l$ the potential will be used only with $z>z_l$. Obviously, the potential has two positive local extrema $z_l=\ln\zeta_l$ and $z_r=\ln\zeta_r$. 

\begin{lemma}\label{lem:inflection}
The potential function $\Phi(z)$ has a unique inflection point in the interval $[z_l,\bar{z}]$. 
\end{lemma}
\begin{proof}
The existence of an inflection point $\zeta_c\in(\zeta_l,\zeta_r)$ follows by Rolle's theorem applied to $\Phi'(z)=3/K_1^2 F(e^z)$ and because $F(\zeta_l)=F(\zeta_r)=0$. The uniqueness of the inflection point $\zeta_c$ follows from the fact that the numerator of $F'(\zeta)$ in (\ref{eq:fderiv}) is for all $\zeta>\zeta_c$ $$P(\zeta)=\zeta^3+\frac{2K_1^2+\zeta_l^3}{\zeta_l}\zeta-4K_1^2>0\ .$$ 
To see this, note first that $P(\zeta_c)=0$. Then the derivative is $P'(\zeta)=3\zeta^2+\frac{2K_1^2+\zeta_l^3}{\zeta_l}>0$, which implies that $P(\zeta)$ is strictly increasing for $\zeta_l<\zeta<\bar{\zeta}$. Thus, $F'(\zeta)>0$ for $\zeta>\zeta_c$ and $F'(\zeta)<0$ for $0<\zeta<\zeta_c$, which proves the uniqueness.
\end{proof}

\begin{remark}\label{rem:inflection}
The inflection point $z_c$ of $\Phi(z)$ is the unique minimum of the function $F(e^z)$ in the interval $(z_l,\bar{z})$.
Thus $F(e^z)\geq F(e^{z_c})$ for all $z_l<z<\bar{z}$. We also have that $F(e^{z_c})=F(\zeta_c)<0$ since $\zeta_c\in (\zeta_l,\zeta_r)$. Moreover, we have $F(\zeta)<0$ for all $\zeta_l<\zeta<\zeta_r$ and $F(\zeta)>0$ for all $\zeta_r<\zeta<\bar{\zeta}$. Since also $F'(\zeta)>0$ for $\zeta>\zeta_c$, we have that $F$ has also a global maximum the point $(\bar{\zeta},F(\bar{\zeta}))$, i.e. $F(e^z)\leq F(e^{\bar{z}})$ for all $z_l<z<\bar{z}$, with $F(e^{\bar{z}})>0$. It is also noted that the point $z_c$ can be computed explicitely (using Mathematica\textsuperscript{\textregistered}) as $z_c=\ln\zeta_c$ with
$$
\zeta_c=\left[\frac{\sqrt[3]{9b+\sqrt{3(4a^3+27b^2)}}}{\sqrt[3]{18}}-\sqrt[3]{\frac{2}{3}}\frac{a}{\sqrt[3]{9b+\sqrt{3(4a^3+27b^2)}}}\right]\ ,
$$
where $a=(2 K_1^2 + \zeta_l^3)/\zeta_l$ and $b=4 K_1^2$.

\end{remark}

Next, we exclude the possibility of non-trivial periodic solutions.
\begin{lemma}\label{lem:nconstper}
The system (\ref{DynSys}) does not have non-trivial periodic solutions (including homoclinic orbits).
\end{lemma}
\begin{proof}
The proof of this lemma is a direct application of the Bendixson-Dulac Theorem \cite{W2003}. Specifically, we have that $\Div \bG(\bu)=-3c/K_1^2<0$ for all $c>0$. Therefore, system (\ref{DynSys}) does not have closed orbits in its phase space.
\end{proof}

By the LaSalle invariance principle, \cite{W2003}, every bounded orbit $\mathcal{R}$ of (\ref{DynSys}) converges to the equilibrium point $(\zeta_l,w_l)$ as $\xi\to-\infty$. Moreover, when $c=0$, the system (\ref{DynSys}) is Hamiltonian and its Hamiltonian coincides with the Liapunov function $H$ of (\ref{Hamqhd}). In such case, it is known that the SGN system of equations possesses classical solitary waves as traveling wave solutions. These are homoclinic orbits, and their speed-amplitude relationship is $s=\sqrt{\bar{\zeta}}=K_1/\zeta_l$. Note that this relationship does not depend on $\delta$ or $\varepsilon$. In the case of the dissipative-dispersive shock waves, their amplitudes should vary between $\zeta_l$ and $(K_1/\zeta_l)^2$.

\begin{remark}
Let $(z,v)$ be a solution to (\ref{DynSys}). If $z(\xi)\to z_l$ and $v(\xi)\to 0$ as $\xi\to-\infty$, then from (\ref{DynSys}) we observe that $z'(\xi)$ and $v'(\xi)$ also tend to 0 as $\xi\to-\infty$. Using induction, we can show that $v^{(j)}(\xi)\to 0$ and $z^{(j+1)}(\xi)=v^{(j)}(\xi)\to 0$ as $\xi\to-\infty$.
\end{remark}

\section{Proof of existence of traveling wave solutions}\label{sec:existence}

Lemma \ref{lem:bounds} verifies that the solution, if it exists, it remains in the strip $\zeta_l<\zeta<(K_1/\zeta_l)^2$ for all $\xi\in\mathbb{R}$. In the dynamical systems language this means that the set 
$$\mathcal{M}=\{(z,v)\in\mathbb{R}^2 : z_l<z<2\ln(K_1/\zeta_l)\}\ ,$$ 
is an invariant set. Since there are no non-trivial periodic solutions to (\ref{DynSys}) due to Lemma \ref{lem:nconstper}, and because both $z'$ and $v'$ tend to 0 at infinity, we conclude that both the $\alpha$-limit set and the $\omega$-limit set must contain critical points of the system (\ref{DynSys}). Since these limit sets are connected by the orbit $\mathcal{R}$, they must each contain exactly one critical point and hence the orbit must tend asymptotically to a critical point both at $+\infty$ and $-\infty$. By Poincar\'{e}-Bendixson Theorem, any bounded orbit $\mathcal{R}$ of (\ref{DynSys}) should connect the critical points $(z_l,v_l)$ and $(z_r,v_r)$: As $\xi\to-\infty$ the orbit connects to $(z_l,v_l)$, while as $\xi\to\infty$ the orbit connects to $(z_r,v_r)$.

Because $(z_l,v_l)=(z_l,0)$ is a saddle point, there are two semi-orbits that converge to it as $\xi\to-\infty$ and approach the equilibrium point tangentially to the stable manifold determined by the slope
$$\lim_{\xi\to-\infty}\frac{v'(\xi)}{z'(\xi)}=\lim_{\xi\to-\infty}\frac{v(\xi)}{z(\xi)}=\lambda_{-}^l\ ,$$ where $\lambda_{-}^l$ as in (\ref{eq:eigs}). Thus, one semi-orbit approaches the saddle point from the region $Q_1=\{(z,v) : z>z_l,v>0\}$ and the other through the region $Q_2=\{(z,v) : z<z_l,v<0\}$. The last case with $z<z_l$ is excluded because of Lemma \ref{lem:bounds}. Thus, the only semi-orbit that can exist is the one with $z>z_l$.

Now we prove that system (\ref{DynSys}) has a unique (up to horizontal translations) solution that satisfies the asymptotic conditions (\ref{eq:limits}) and (\ref{eq:lc}).

\begin{theorem}\label{thm:exist}
Let $\delta>0$, $\varepsilon>0$ and $s>\sqrt{\zeta_l}+w_l/\zeta_l$ given constants. There is a unique, global solution $(u,v)$ of the system (\ref{DynSys}), which defines in the phase space a bounded orbit $\mathcal{R}\subset \mathcal{M}$ that tends to $(z_l,v_l)=(z_l,0)$ as $\xi\to-\infty$ and $(z_r,v_r)=(z_r,0)$ as $\xi\to\infty$, subject to the Lax condition (\ref{eq:lc}).
\end{theorem}
\begin{proof}
Let $(z(\xi),v(\xi))$ be a solution of (\ref{DynSys}) that corresponds to one of the two semi-orbits lying on the stable manifold of the saddle-point $(z_l,v_l)=(z_l,0)$, and assume that it is defined for at least large values of $\xi$ for the particular value of $s$. Such a solution exists since the function $\bG(\bu)$ is locally Lipschitz. Specifically, if $v\in\mathbb{R}$ and $z\in (z_l,2\ln(K_1/\zeta_l))$ we have that $\bG$ is continuously differentiable with
$$\nabla\bG(\bu^\ast)\bu=\left(v,~-\frac{3}{K_1^2}\left[ze^{z^\ast}-K_1^2\frac{2z}{e^{2z^\ast}}+(K_2-sK_1)\frac{z}{e^{z^\ast}}+cv \right] \right)^T\ .$$
The standard theory of ordinary differential equations guarantees that there is a unique (subject to horizontal translations) solution of system (\ref{DynSys}) for as long as the solution remains bounded. By Lemma \ref{lem:bounds} the orbit will remain in $\mathcal{M}$, and particularly we will have $z_l<z<\ln(K_1/\zeta_l)$. To complete the proof of global existence, it remains to show that $v$ remains bounded for all values of $\xi$. 

First we prove that $v(\xi)>-2F(e^{\bar{z}})/c$ for all $\xi\in\mathbb{R}$. Since $v(\xi)\to 0$ as $\xi\to-\infty$, we have that $v(\xi)>-2F(e^{\bar{z}})/c$ for $\xi<M$ for some $M<0$. Recall that from Remark \ref{rem:inflection}) we have that $F(e^{\bar{z}})>0$. For contradiction, assume that there is a minimal moment $\xi_0=\min \{\xi: v(\xi)=-\mu\leq -2F(e^{\bar{z}})/c\}$, then $v'(\xi_0)\leq 0$. On the other hand, taking into account the Remark \ref{rem:inflection} we have that $F(e^{\bar{z}})$ is the maximum value of the function $F(e^z)$ for $z_l<z<\bar{z}$. Therefore, from the second equation of (\ref{DynSys}) we have that
$$
\begin{aligned}
v'(\xi_0)&=-\frac{3}{K_1^2}F(e^{z(\xi_0)})-\frac{3c}{K_1^2}v(\xi_0) =-\frac{3}{K_1^2}F(e^{z(\xi_0)})+\frac{3c}{K_1^2}\mu\\
&\geq -\frac{3}{K_1^2}F(e^{\bar{z}})+2\frac{3}{k_1^2}F(e^{\bar{z}})=\frac{3}{K_1^2}F(e^{\bar{z}})>0\ .
\end{aligned}$$
This is a contradiction and thus $v(\xi)>-2F(e^{\bar{z}})/c$.

We will also prove that 
\begin{equation}\label{eq:vineq}
v(\xi)<-2F(e^{z_c})/c \quad \text{for all $\xi\in\mathbb{R}$}\ .
\end{equation} By the Remark \ref{rem:inflection} we have that $F(e^{z_c})<0$ is the minimum value of $F(e^z)$ for $z_l<z<\bar{z}$ and thus $-F(e^z)\leq -F(e^{z_c})$ for $z_l<z<\bar{z}$. Since $v(\xi)\to 0$ as $\xi\to-\infty$, we assume that there is a minimal moment $\xi_1$ such that $\xi_1=\min\{\xi: v(\xi)=\nu\geq -2F(e^{z_c})/c>0\}$. For that particular moment we will have $v'(\xi_1)\geq 0$ since $v(\xi)\to0$ as $\xi\to-\infty$. Again, from the second equation of (\ref{DynSys}) we have that
$$
\begin{aligned}
v'(\xi_1)&=-\frac{3}{K_1^2}F(e^{z(\xi_1)})-\frac{3c}{K_1^2}v(\xi_1) =-\frac{3}{K_1^2}F(e^{z(\xi_0)})-\frac{3c}{K_1^2}\nu\\
&\leq -\frac{3}{K_1^2}F(e^{z_c})+2\frac{3}{k_1^2}F(e^{z_c})=\frac{3}{K_1^2}F(e^{z_c})<0\ .
\end{aligned}$$
And this completes the proof of global existence.
\end{proof}

In the next section we study the shape of the traveling waves in detail.

\section{Oscillatory and regularized shock wave solutions}\label{sec:shapes}

In the previous section, we showed that given $\delta>0$, $\varepsilon>0$, $\zeta_l> 0$, and $w_l$, there is a unique (up to horizontal translations) bounded orbit $\mathcal{R}$ connecting the point $(\zeta_l,w_l)$ with the point $(\zeta_r,w_r)$ satisfying (\ref{eq:leftlims2}). In Section \ref{sec:dynamics}, we saw that the nature of the equilibrium point $(\zeta_r,w_r)$ can change depending on the choice of the parameters $\delta$, $\varepsilon$, $\zeta_l$, and $w_l$. Specifically, we saw that when $c^2< \tfrac{4}{3}K_1^2F'(\zeta_r)\zeta_r$, the equilibrium is a stable spiral; otherwise, it is a stable node. We will show that, indeed, in these cases, the traveling waves are either oscillatory shock waves or regularized shock waves, and we will investigate their profiles. We begin with the second case:

\begin{theorem}\label{thm:regshockw}
Let $z(\xi)$ be the unique solution to (\ref{ODE2z}) for values of $\varepsilon$, $\delta$, $s$, $\zeta_l$ and $w_l$ such that 
\begin{equation}
    \varepsilon^2\geq\frac{4\delta K_1^2 }{3s^2}F'(\zeta_r)\zeta_r\ .
\end{equation}
Then for all $\xi\in\mathbb{R}$, $z_l<z(\xi)<z_r$ and $z'(\xi)>0$. Moreover, $z(\xi)$ has a unique inflection point $\xi_1$ such that $(\xi-\xi_1)z''(\xi_1)<0$ for all $\xi\not=\xi_1$.
\end{theorem}
\begin{proof}
Let $(z,v)$ be the corresponding solution of the system (\ref{DynSys}) under the assumptions of the Theorem \ref{thm:exist}. To prove that $z'(\xi)>0$ for the particular solution, it is sufficient to show that $z(\xi)>z_l$ and $v(\xi)>0$ since $v(\xi)=\zeta'(\xi)$. Thus, the orbit $\mathcal{R}$ should remain in the region $Q_1=Q_1=\{(z,v) : z>z_l,v>0\}$. 

We have already shown that $z_l<z(\xi)<\bar{z}$. It remains to show that the orbit $\mathcal{R}$ cannot exit the semi-infinite strip
$$\mathcal{S}=\{(z,v)~:~v>0, ~z_l<z<z_r \}\ ,$$
where $\zeta_r$ as in (\ref{eq:leftlims2}). 

We will prove a more useful enclosure of the orbit $\mathcal{R}$ inside the triangle $\mathcal{T}$ defined by the line segments
$$\ell_0=\{(z,v)~ : ~ v=0,~ z_l\leq z\leq z_r\}, \quad \ell_1=\{(z,v)~:~ v=m(z-z_r),~z_l\leq z\leq z_r\}\ ,$$
and
$$\ell_2=\{(z,v)~:~ 0\leq v\leq m(z_l-z_r),~ z = z_l \}\ ,$$
with $$m=\lambda_{-}^r<0\ .$$
In other words we will show that the domain enclosed by the triangle $\mathcal{T}$ is a trapping region \cite{W2003}, cf. Figure \ref{fig:phase}.

\begin{figure}[ht!]
  \centering
\includegraphics[width=\columnwidth]{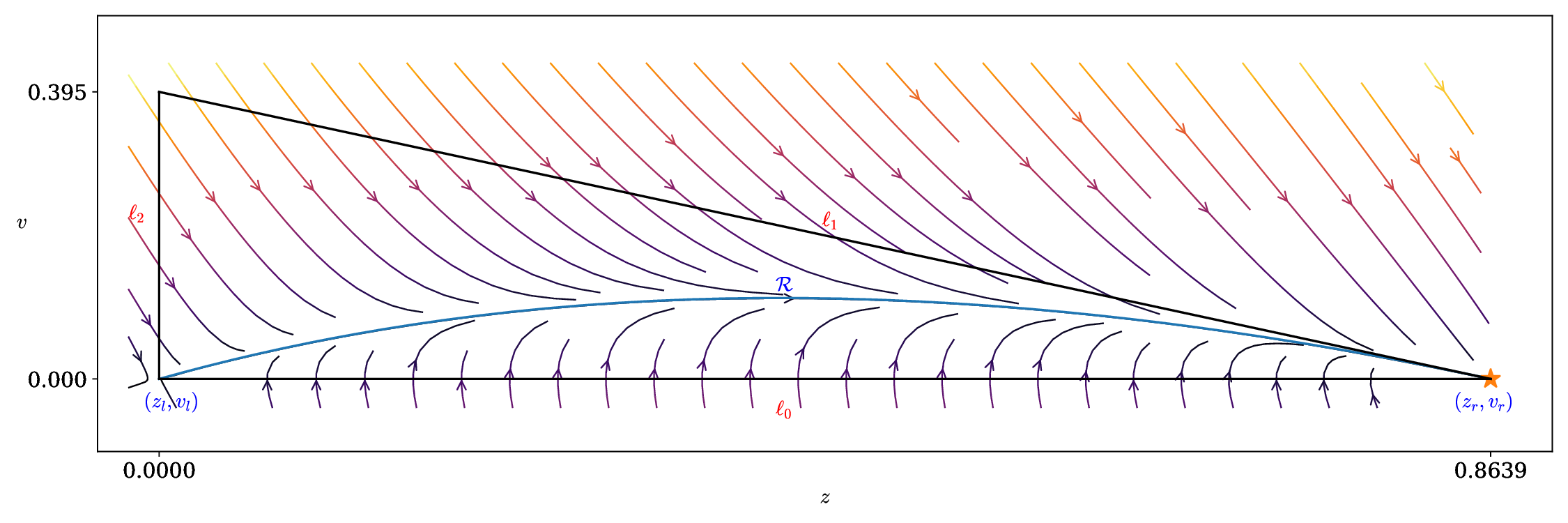}
  \caption{The phase-space and the corresponding streamlines of the flow for a dissipative dominated flow ($s=2$, $\zeta_l=1$, $w_l=0$, $\varepsilon=0.3$, $\delta=0.02$)}
  \label{fig:phase}
\end{figure}

We have already excluded the possibility of the orbit $\mathcal{R}$ to intersect the line $\ell_2$ since $z_l<z$. We will next show that the orbit $\mathcal{R}$ does not intersect the line segment $\ell_0$. Since the orbit $\mathcal{R}$ approaches the point $(z_l,0)$ from the region $Q_1=\{(z,v)~:~z>z_l, ~v>0 \}$ as $\xi\to-\infty$, we assume (for contradiction) that the orbit $\mathcal{R}$ intersects the segment $\ell_0$, and that the $\xi_0$ is the smallest value $\xi\in\mathbb{R}$ such that $v(\xi_0)=0$ and $z_l<z(\xi_0)< z_r$, and $v(\xi)>0$ for all $\xi<\xi_0$. (For $\xi>\xi_0$ the function $v(\xi)$ can be anything, positive-negative-zero). This means that $v'(\xi_0)=z''(\xi_0)\leq 0$ ($z$ should be concave down). From the second equation of (\ref{DynSys}), for $z_l<z(\xi_0)<z_r$ we have 
\begin{equation}\label{eq:vderiv}
v'(\xi_0)=-\frac{3}{K_1^2}F(e^{z(\xi_0)})=-\frac{3}{K_1^2}\frac{(\zeta(\xi_0)-\zeta_l)(\zeta(\xi_0)-\zeta_r)(\zeta(\xi_0)+\zeta_l+\zeta_r)}{2\zeta^2(\xi_0)}>0 ,
\end{equation}
which contradicts with the assumption of the existence of the point $\xi_0$ and thus the orbit $\mathcal{R}$ cannot intersect the segment $\ell_0$.

We will now prove that the orbit $\mathcal{R}$ cannot intersect also the segment $\ell_1$. We know that the orbit is inside the triangle $\mathcal{T}$ for $\xi\to-\infty$ since it approaches $(z_l,0)$ through the region $Q_1$. Assume for contradiction that the orbit $\mathcal{R}$ intersects the segment $\ell_1$ and let $\xi_0$ be the smallest value such that $(z(\xi_0),v(\xi_0))\in\ell_1$. It follows that $z_l< z(\xi_0)< z_r$ and (since $m<0$)
$$z'(\xi_0)=v(\xi_0)=m(z(\xi_0)-z_r)>0\ .$$
Thus, the orbit $\mathcal{R}$ at the point of intersection must have $v'(\xi_0)/z'(\xi_0)\geq m$, because otherwise $z'(\xi_0)<0$. Dividing the equations of (\ref{DynSys}) we have
$$\begin{aligned}
m\leq \frac{v'(\xi_0)}{z'(\xi_0)} &=-\frac{3}{K_1^2}\frac{F(e^{z(\xi_0)})+c v(\xi_0)}{v(\xi_0)}\\
&= -\frac{3}{K_1^2}\left[\frac{F(e^{z(\xi_0)})}{v(\xi_0)} +c \right] \\
&=-\frac{3}{K_1^2}\left[\frac{F(e^{z(\xi_0)})}{m(z(\xi_0)-z_r)} +c \right]\ .
\end{aligned}$$
Equivalently, since $m=\lambda_{-}^r$ and from the characteristic equation (\ref{eq:characteristiceq}) we have
$$ 0\leq m^2+\frac{3c}{K_1^2} m+\frac{3}{K_1^2}\frac{F(e^{z(\xi_0)})}{z(\xi_0)-z_r}=\frac{3}{K_1^2}\left[\frac{F(e^{z(\xi_0)})}{z(\xi_0)-z_r}-F'(e^{z_r})e^{z_r}\right]\ .$$
On the other hand, we have
$$
\begin{aligned}
0<\frac{F(e^{z(\xi_0)})}{z(\xi_0)-z_r} &= \frac{F(e^{z(\xi_0)})-F(e^{z_r})}{z(\xi_0)-z_r}\\
&=F'(e^\psi)e^\psi \quad \text{for some $z(\xi_0)<\psi<z_r$ by the mean-value theorem}\\
&<F'(e^{z_r})e^{z_r}\ .
\end{aligned}
$$
To see the last inequality, first compute the second derivative of $F$
$$F''(\zeta)=-\frac{2 K_1^2(\zeta-3\zeta_l)+\zeta\zeta_l^3}{\zeta^4\zeta_l}=\frac{3\zeta_l\zeta_r(\zeta_l+\zeta_r)-\zeta(\zeta_l^2+\zeta_l\zeta_r+\zeta_r^2)}{\zeta^4}\ .$$
This implies that $F''(\zeta_l)>0$ and $F''(\zeta_r)>0$.
Moreover, $F''(\zeta)=0$ has only the root
$$
\tilde{\zeta}=\frac{3\zeta_l\zeta_r(\zeta_l+\zeta_r)}{\zeta_l^2+\zeta_l\zeta_r+\zeta_r^2}\ ,$$ and is continuous.
Thus $F''(\zeta)>0$ for all $\zeta_l<\zeta<\zeta_r$ and thus $F'$ is increasing in $(\zeta_l,\zeta_r)$.

The fact that the triangle $\mathcal{T}$ is a trapping region can also be inferred by studying the flow through its edges. As can be observed in Figure \ref{fig:phase}, the triangle is indeed a trapping region, and no orbit can escape but rather converges to $(z_l,0)$ through the region $Q_1$. Of course, the detailed analysis of the flow requires arguments similar to those presented earlier and is omitted.

Since we proved that $z(\xi)$ is increasing from the point $(z_l,0)$ to $(z_r,0)$, and $v(\xi)>0$ for all $\xi\in\mathbb{R}$ and $v(\xi)\to 0$ as $\xi\to\pm\infty$, it is implied by the mean value theorem that there is a $\xi_1\in\mathbb{R}$ such that $v'(\xi_1)=z''(\xi_1)=0$. If $v'(\xi)=0$, then the second equation of (\ref{DynSys}) yields 
$$v''(\xi)=-\frac{3}{K_1^2}F'(e^{z(\xi)})e^{z(\xi)}z'(\xi)\ .$$
Since $z_l<z(\xi)<z_r$ and $z'(\xi)>0$ for all $\xi$, we have three possibilities:
\begin{enumerate}
    \item [(i)] If $z(\xi)<z_c$, then $F'(e^{z(\xi)})<0$, i.e. $v''(\xi)>0$ and thus $\xi_1$ is a strict local minimum of $v$
    \item [(ii)] If $z(\xi)>z_c$, then $F'(e^{z(\xi)})>0$, i.e. $v''(\xi)<0$ and thus $\xi_1$ is a strict local maximum of $v$
    \item [(iii)] If $z(\xi)=z_c$, then $F'(e^{z(\xi)})=0$, i.e. $v''(\xi)=0$ and thus $\xi_1$ is a saddle point of $v$
\end{enumerate}
Assume that $\xi_1$ is a maximum point of $v$ and also for contradiction assume that there is another local maximum point at $\xi_1'\not=\xi_1$. Without loss of generality we assume that $\xi_1'<\xi_1$. By the mean-value theorem we have that there is a local minimum $\bar{\xi}$ with  $\xi_1'<\bar{\xi}<\xi_1$. Since $z(\xi)$ is increasing, we have $z_c<z(\xi'_1)<z(\bar{\xi})<z(\xi_1)$ (by (ii) since $\xi_1'$ is a local maximum), which contradicts with statement (i) which ensures that the local minimum will be $z(\bar{\xi})<z_c$. If $z(\xi)<z_c$, this means that $\xi$ is strictly local minimum and $z'(\xi)=v(\xi)>0$. Since $v(\xi)\to 0$ as $\xi\to-\infty$, there will be another local maximum $\xi_1''<\xi$ with $z(\xi_1'')<z(\xi)<z_c$ which contradicts with the statement (ii) and ensures that $z(\xi_1'')>z_c$. Finally, if $v'(\xi)=0$ and $z(\xi)=z_c$, then from the second equation of (\ref{DynSys}) we have that
$$v(\xi)=-\frac{F(e^{z_c})}{c}>0\ ,$$
but this is in contrast with (\ref{eq:vineq}) (see the proof of Theorem \ref{thm:exist}) where we showed that $v(\xi)<-\frac{F(e^{z_c})}{c}$ for all $\xi\in\mathbb{R}$, and thus $\xi$ cannot be a saddle point. Hence, it can only be $v'(\xi_1)=z''(\xi_1)=0$ and $z_c<z(\xi_1)<z_r$. Since $v(\xi)>0$ for all $\xi$ and $v(\xi)\to 0$ as $\xi\to \pm\infty$, then $v'(\xi)>0$ for $\xi<\xi_1$ and $v'(\xi)<0$ for $\xi>\xi_1$. Therefore, there is a unique inflection point $\xi_1$ of $z$ such that $(\xi-\xi_1)z''(\xi)<0$ for $\xi\not=\xi_1$.
\end{proof}

We turn now our attention to the regime where the dispersion dominates the dissipation with $$\varepsilon^2<\frac{4\delta K_1^2}{3c^2}F'(\zeta_r)\zeta_r\ .$$ We call this regime, the regime of moderate dispersion. In such case the traveling wave consists of decreasing undulations as $\xi\to\infty$ and is lead by a traveling front. Specifically, we have the following theorem:

\begin{theorem}
Let $z(\xi)$ be the unique (up to horizontal translations) solution to (\ref{DynSys}) and $\varepsilon$ is such that $\varepsilon^2<4\delta K_1^2F'(\zeta_r)\zeta_r/3c^2$. Then for all $\xi\in\mathbb{R}$ we have that $z(\xi)>z_l$ and:
\begin{enumerate}
    \item[(a)] If $M_0=\sup_\xi z(\xi)$, then $M_0$ is attained at a unique value $\xi=\chi_0$, and for $\xi<\chi_0$, $z'(\xi)>0$.
    \item[(b)] There is $\xi_0<\chi_0$ such that $(\xi-\xi_0)z''(\xi)<0$, for all $\xi<\chi_0$, $\xi\not=\xi_0$.
    \item[(c)] The solution $z(\xi)$ has an infinte number of local maxima and minima. These are taken on at points $\{\chi_i\}_{i=0}^\infty$ and $\{\omega_i\}_{i=0}^\infty$ where $\chi_i<\omega_{i+1}<\chi_{i+1}$ for all $i\geq 0$, and $\lim_{i\to\infty} \chi_i=\lim_{i\to\infty}\omega_i=\infty$. Moreover, $z_r<z(\chi_{i+1})<z(\chi_i)$ and $z_r>z(\omega_{i+1})>z(\omega_i)$.
\end{enumerate}
\end{theorem}
\begin{proof}
The fact that the traveling wave is $z(\xi)>z_l$ was shown in Lemma \ref{lem:bounds}. 
Since $(z_r,0)$ is a spiral point in the case of $\varepsilon^2<4\delta K_1^2F'(\zeta_r)\zeta_r/3s^2$ we expect that the solution will be oscillatory as $\xi\to +\infty$, and in particular of the form
\begin{equation}\label{eq:harman}
(z(\xi)-z_l,v(\xi))=Ce^{\mu \xi}(\cos(\nu\xi+\theta_0+o(1)),\sin(\nu\xi+\theta_0+o(1)))\ ,
\end{equation}
as $\xi\to+\infty$, where $\mu={\rm Re} (\lambda^r_+)$ and $\nu={\rm Im}(\lambda^r_+)\not=0$, with $\lambda^r_+$ as defined in (\ref{eq:eigs}), while $C$ and $\theta_0$ depend on the solution \cite{Hartman2002}. 
Therefore, there are infinite many points where $v(\xi)$ vanishes. According to Lemma \ref{lem:locextr} these points are isolated and are strict local maxima or minima of the solution $z$. Moreover, according to the same lemma, a local maximum corresponds to a point $\xi$ where $z(\xi)>z_l$, and a local minimum to a point $\xi$ where $z_l<z(\xi)<z_r$. 

Since $(z,v)\to(z_l,0)$ as $\xi\to-\infty$ and $z(\xi)>z_l$ for all $\xi$ we have that $v(\xi)>0$ for $\xi\to-\infty$ and thus the orbit $\mathcal{R}$ lies in the region $Q_1=\{(x,v)~:~z>z_l,v>0\}$ for large values of $\xi$, while the orbit cannot intersect the segment $\ell_0=\{(z,v):v=0,~ z_l\leq z\leq z_r\}$ from $Q_1$. 
To see this (as in the proof of Theorem \ref{thm:regshockw}) assume that there is a point $\xi_0$ such that $(z(\xi_0),v(\xi_0))\in \ell_0$, i.e. $v(\xi_0)=0$ and $z(\xi_0)\leq z_r$, and $v(\xi)<0$ for $\xi>\xi_0$. This means that the point $z(\xi_0)$ is concave down with $z''(\xi_0)=v'(\xi_0)\leq 0$. From the second equation of (\ref{DynSys}) and similarly to (\ref{eq:vderiv}) we have that
$$v'(\xi_0)=-\frac{3}{K_1^2}F(e^{z(\xi_0)})>0\ ,$$
since $z_l< z(\xi)< z_r$, which is a contradiction. Thus, the orbit $\mathcal{R}$ must exit the region $Q_1$ at a point $\chi_0$ where $z(\chi_0)>z_r$ and $v(\chi_0)=0$. And $\chi_0$ will be a strict local maximum of $z$ and $z'(\xi)>0$ for $\xi<\chi_0$ which proves (a). Following the same steps as in the proof of Theorem \ref{thm:regshockw} we conclude that there is a point $\xi_0<\chi_0$ where $z''(\xi_0)=0$ and $z''(\xi)>0$ for $\xi<\xi_0$, and $z''(\xi)<0$ for $\xi>\xi_0$, which proves argument (b).

Let $\chi_1>\chi_0$ is the next local maximum of $z$ and let $\omega_1$ be the unique local minimum in the interval $(\chi_0,\chi_1)$. Inductively we define the increasing sequences $\{z_i\}_{i=0}^\infty$ and $\{\omega_i\}_{i=1}^\infty$ of local maxima and minima, respectively. Due to Lemma \ref{lem:locextr} we have that $z_r<z(\chi_i)<\bar{z}$ and $z_l<z(\omega_i)<z_r$. Moreover, $z(\chi_{i+1})<z(\chi_i)$ and $z(\omega_{i+1})>z(\omega_i)$ for all $i$ because otherwise the orbit $\mathcal{R}$ will intersect itself which is impossible since system (\ref{DynSys}) is autonomous, \cite{W2003}. For the same reason, $\omega_i$ and $\chi_i$ must accumulate at $+\infty$ and the solution cannot become constant $z_l$ in finite time.
\end{proof}

\begin{figure}[ht!]
  \centering
\includegraphics[width=\columnwidth]{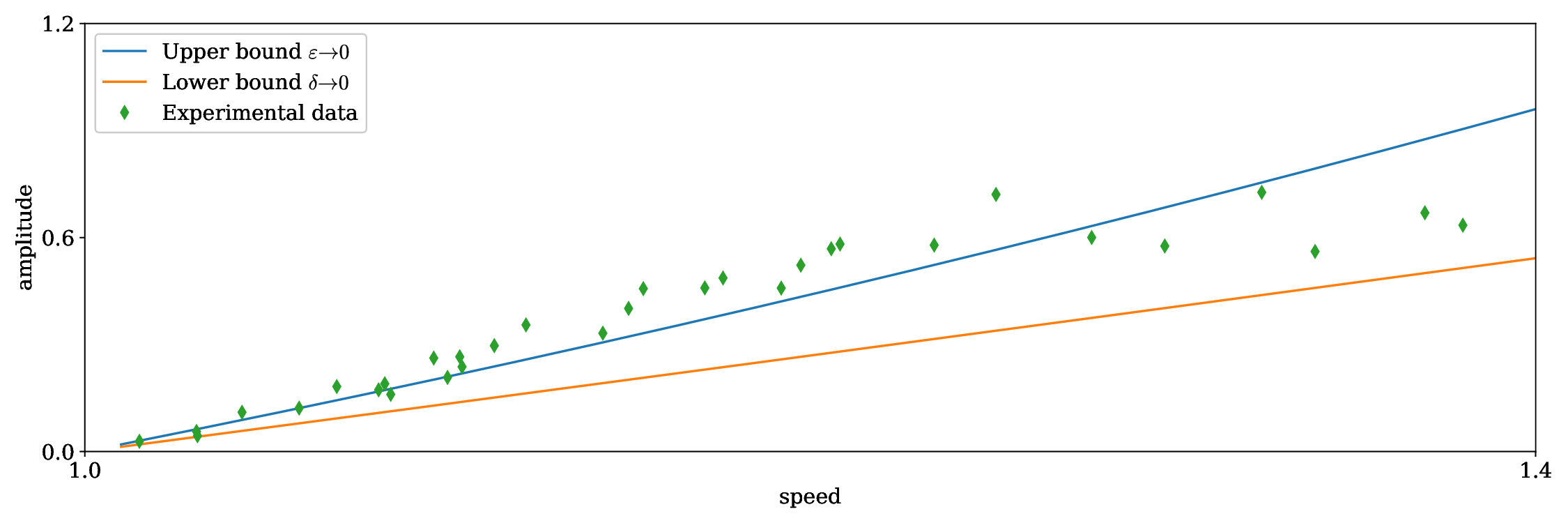}
  \caption{Amplitudes of undular bores observed in  \cite{Favre1935}, \cite{SZ2002} and \cite{T1994} {\em vs} speed $s$ in comparison with the extreme values of the amplitude $A$ as functions of time as estimated by the modified SGN equations}
  \label{fig:compare1}
\end{figure}

\begin{figure}[ht!]
  \centering
\includegraphics[width=\columnwidth]{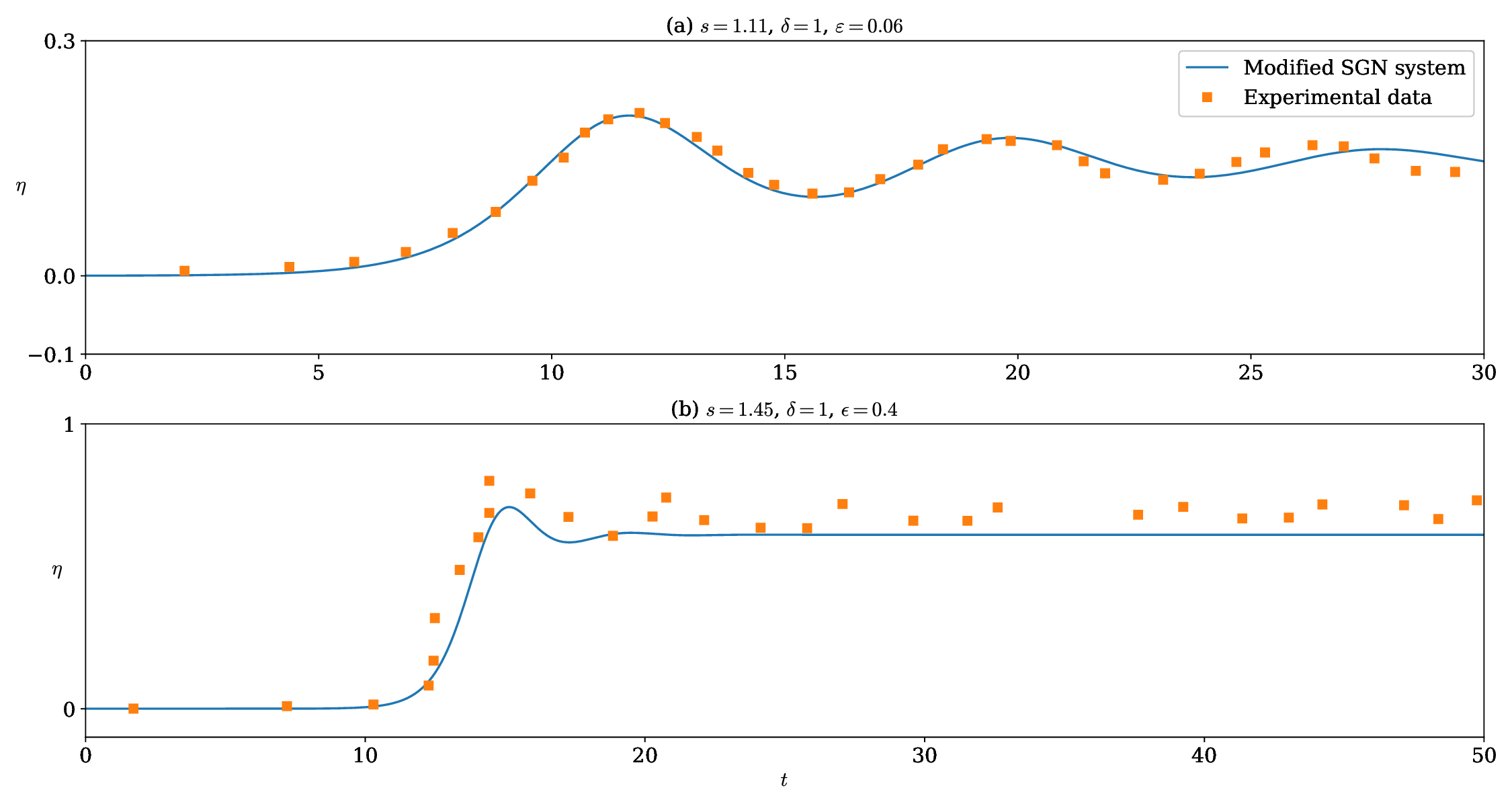}
  \caption{Oscillatory shock waves in comparison with experimental data of \cite{Chanson2010}}
  \label{fig:experiment12}
\end{figure}

% \begin{figure}[ht!]
%   \centering
% \includegraphics[width=\columnwidth]{Figures/breaking.eps}
%   \caption{Comparison with laboratory data for the dam-break problem \cite{OK2010}}
%   \label{fig:break}
% \end{figure}

We showed that the traveling wave solutions of the system (\ref{Serre}) have amplitude $A$ such that 
$$\zeta_r\leq A \leq K_1^2/\zeta_l^2\ ,$$
where the lower bound $\zeta_l$ and the upper bound $K_1^2/\zeta_l^2$ can be computed using the formulas (\ref{eq:leftlims2}) and (\ref{eq:defk1}).
Taking $\zeta_l=1$, i.e. considering waves on a zero mean water level, we compare, in Figure \ref{fig:compare1}, the observed amplitudes reported in various laboratory experiments by \cite{Favre1935}, \cite{SZ2002}, and \cite{T1994} with the extreme values of the amplitude $A$ as functions of time.

We observe that the predicted upper bound for the amplitude of the traveling wave (which coincides with the amplitude of the solitary wave of the classical SGN equations) is underestimated slightly compared to the experimental values for speeds between, around, $1.1$ and $1.25$, while it predicts precisely well the corresponding amplitudes for traveling waves with speeds less than $1.1$. For traveling waves of speed greater than $1.25$, dissipation becomes important and the values fall in the range where $\delta,\varepsilon$ are both positive. It is worth mentioning that in a similar comparison between a weakly nonlinear and weakly dispersive Boussinesq-Peregrine system with the same experimental data  in \cite{BMT2022}, the corresponding upper bound was slightly more accurate compared to the one presented here. This indicates that for the particular values, the small amplitude assumption is not restrictive. Nevertheless, the modified SGN equations  approximate sufficiently undular bores. 

In Figure \ref{fig:experiment12} we compare numerically generated oscillatory shock waves with experimental data of \cite{Chanson2010}. Specifically, we solve the system of ordinary differential equations (\ref{DynSys}) with appropriate initial conditions so as to ensure that the solution remains on the unstable manifold passing through the origin. In Figure \ref{fig:experiment12}(a) we present an undular bore with speed $s=1.11$ where the dissipation is not very important. On the contrary, in Figure \ref{fig:experiment12}(b) we present an undular bore with speed $s=1.45$ where is breaking and the dissipation due to turbulence is very important. For these reasons we consider the parameters $\varepsilon=0.06$ and $0.4$ so as to match the numerical solutions with the experimental data. The choice of these parameters is in agreement with the corresponding parameters used in \cite{BMT2022} in the case of the asymptotically derived dissipative Peregrine system. It is worth noting that the modified SGN equations underestimate slightly the value $\zeta_r$, concluding that for particular cases the weakly nonlinear dissipative Boussinesq-Peregrine system seems to be in better agreement with the experimental values than the SGN equations. 

% Finally we present a comparison with some laboratory data from \cite{OK2010}. Here we used as initial conditions for the velocity $u(x,0)=0$ and for $h(x,0)$ smooth Riemann initial data appropriate for the dam break problem. The initial condition generates a rarefaction wave that propagates to the left and a wave that is similar to an oscillatory shock wave that travels to the right.

% Although the wave front of the numerical solution is lagging the experimental one, it is very close to the laboratory data. The difference between the solution presented in \cite{CC2021} and the present one in this experiment seem to be of similar magnitude. The overall  agreement in the approximation of the breaking waves of this experiment is satisfactory, and thus this experiment suggest that the modified SGN system  can be used to predict quite accurately even a breaking undular bore.

\section{Some limiting cases}\label{sec:comparison}

\subsection{Waves in the absence of dissipation}\label{sec:nodisip}

In this section we compare the solutions of the modified SGN equations (\ref{eq:Serre1}) for $\varepsilon>0$ with the corresponding solution of the classical SGN equations ($\varepsilon=0$) that emerge from the same initial data. This comparison reveals that the difference between the two solutions is of order $\varepsilon t$. Therefore, for appropriate time scales and sufficiently small values of $\varepsilon$, the solutions of the classical SGN equations can be approximated by the solutions of the modified SGN system or the other way around. This observation explains why the dispersive shock waves emerging from smooth Riemann data are similar to oscillatory shock waves. Moreover, it justifies the use of the modified SGN equations for small values of $\varepsilon$. 

The proofs presented in this section rely on two lemmata from \cite{AB1991} which we mention here for the sake of completeness. The first lemma provides a useful inequality:
\begin{lemma}{(Lemma 1 of \cite{AB1991})}\label{lem:lemmab1}
Let $k\geq 3$ be an integer. Then there exists a constant $C=C(k)>0$ such that for any $y>0$ and $\epsilon$ in the interval $(0,1)$, the inequality
\begin{equation}
\epsilon y+y^2+\epsilon^{-1/2}y^3+\epsilon^{-1}y^4+\cdots +\epsilon^{(2-k)/2}y^k\leq C(\epsilon y+\epsilon^{2-k}y^k)\ ,
\end{equation}
is valid.
\end{lemma}
Here we will use this lemma for $k=3$. The second lemma, again adapted to our needs, is a Gronwall's type inequality.
\begin{lemma}{(Lemma 2 of \cite{AB1991})}\label{lem:lemmab2}
Let $A, B>0$ be given. Then there is a maximal time $T$ and a constant $C>0$ independent of $A$ and $B$ such that if $y(t)$ is any non-negative differentiable function defined on $[0,T]$ satisfying
$$\begin{aligned}
&\frac{d}{dt}y^2(t)\leq A y(t) + B y^3(t),\qquad 0\leq t\leq T\ ,\\
&y(0)=0\ ,
\end{aligned}$$
then 
$$y(t)\leq CA t, \quad \text{ for }\quad 0\leq t\leq T\ .$$
\end{lemma}

Let $h^{\varepsilon}$ and $q^\varepsilon=h^\varepsilon u^\varepsilon$ denote the solution of the system (\ref{eq:Serre2}) for $\varepsilon\geq 0$, emerging from the same initial conditions. Assume that the initial conditions are appropriately smooth so as the integrals in the following expressions make sense and also guarantee the existence of unique solution to both systems. Denote $e^h=h^\varepsilon-h^0$ and  $e^u=u^\varepsilon-u^0$. Then for any value $\varepsilon \geq 0$ we have the following comparison theorem:
\begin{theorem}
Assume that for the initial conditions $h^\varepsilon(x,0)=h(x,0)$ and $u^\varepsilon(x,0)=u(x,0)$ the solutions $(h^\varepsilon, u^\varepsilon)\in C^1([0,T],W^1_\infty)\times C^1([0,T],W^2_\infty)$ are uniformly bounded solutions of the system (\ref{eq:Serre1}) for all $\varepsilon\in [0,1)$ with bounds independent of $\varepsilon$. We further assume that the solutions satisfy the non-cavitation assumption (\ref{eq:cavity}) for all $\varepsilon\in [0,1)$. If $$\mathcal{E}(t)=\sqrt{\|e^h\|^2+\|e^u\|^2_1}\ ,$$ then there is a constant $C$ independent of $\varepsilon$ such that
\begin{equation}
   \mathcal{E}(t)\leq C\varepsilon t\ ,
\end{equation}
for all $t\in [0,T]$.
\end{theorem}
\begin{proof}
To simplify the notation we will use the symbol $\lesssim$ to denote $\leq C$ for a generic constant $C$ independent of $\varepsilon$. 

First note that subtracting the equations of system (\ref{eq:Serre1}) corresponding to the solutions $(h^0,u^0)$ and $(h^\varepsilon,u^\varepsilon)$ for $\varepsilon>0$ yields the following equations for the error functions $e^h$ and $e^u$:
\begin{equation}\label{eq:syserrf}
\begin{aligned}
&e^h_t+[h^\varepsilon u^\varepsilon-h^0 u^0]_x=0\ ,\\
&h^\varepsilon u^\varepsilon_t-h^0 u^0_t+\tfrac{1}{2}[(h^\varepsilon)^2 - (h^0)^2]_x+h^\varepsilon u^\varepsilon u^\varepsilon_x-h^0u^0u^0_x-\tfrac{\delta}{3}[\gamma^\varepsilon-\gamma^0]_x=\varepsilon [h^\varepsilon u^\varepsilon]_{xx}\ ,
\end{aligned}
\end{equation}
where the second equation occurs after multiplication of the corresponding equations of (\ref{eq:Serre1}) with $h^\varepsilon$ and $h^0$, respectively.

To treat the nonlinear term of the first equation of (\ref{eq:syserrf}) we observe
\begin{equation}\label{eq:firsten1}
h^\varepsilon u^\varepsilon -h^0 u^0 = h^\varepsilon e^u+u^\varepsilon e^h-e^he^u\ .
\end{equation}
Thus, we rewrite the first equation of (\ref{eq:syserrf}) as
$$e^h_t+(h^\varepsilon e^u+u^\varepsilon e^h-e^h e^u)_x=0\ ,$$
which after multiplication with $e^h$ and integration leads to
$$
\tfrac{1}{2}\frac{d}{dt}\|e^h\|^2 + \left((h^\varepsilon e^u)_x,e^h\right) +\tfrac{1}{2}\left(u^\varepsilon_xe^h,e^h \right) -\left((e^h e^u)_x,e^h \right) =0  \ .
$$
Further to that, we write $\left((e^h e^u)_x,e^h \right)=\tfrac{1}{2}\left((e^h)^2,e^u_x\right)$ leading to 
\begin{equation}\label{eq:firsten}
\tfrac{1}{2}\frac{d}{dt}\|e^h\|^2 + \left((h^\varepsilon e^u)_x,e^h\right) +\tfrac{1}{2}\left(u^\varepsilon_x,(e^h)^2 \right) +\tfrac{1}{2}\left((e^h)^2,e^u_x\right) =0  \ .
\end{equation}
Similarly, we write the second equation of (\ref{eq:syserrf}) after multiplication with $e^u$ and integration over $\mathbb{R}$ in the form
\begin{equation}\label{eq:compeq1}
(h^\varepsilon u^\varepsilon_t-h^0 u^0_t,e^u)+\tfrac{\delta}{3}\left(\gamma^\varepsilon-\gamma^0,e^u_x \right)=\tfrac{1}{2}\left([h^\varepsilon]^2-[h^0]^2,e^u_x \right) -\left(h^\varepsilon u^\varepsilon u^\varepsilon_x-h^0 u^0 u^0_x,e^u \right)-\varepsilon(q^\varepsilon_x,e^u_x)\ .
\end{equation}

In the sequel we will denote by $C^k_\varepsilon$ for $k=1,2,\dots$, generic constants that depend on $\varepsilon$. We treat the terms in the previous sums separately: First, we consider the first term of the right-hand side of (\ref{eq:compeq1}) 
\begin{equation}\label{eq:hineq1}
\Omega_1=\tfrac{1}{2}([h^\varepsilon]^2-[h^0]^2,e^u_x)=\tfrac{1}{2}(e^h(h^\varepsilon+h^0),e^u_x)\leq C \|h^\epsilon+h^0\|_\infty\|e^h\|\|e^u_x\|\lesssim \|e^h\|^2+\|e^u_x\|^2\ .
\end{equation}
Furthermore, using the relation (\ref{eq:firsten1}) we have that
$$
\begin{aligned}
h^\varepsilon u^\varepsilon u^\varepsilon_x-h^0 u^0 u^0_x &=h^\varepsilon u^\varepsilon e^u_x+(h^\varepsilon u^\varepsilon -h^0 u^0)u^0_x\\
&=h^\varepsilon u^\varepsilon e^u_x+(h^\varepsilon e^u+u^\varepsilon e^h-e^he^u)u^0_x\\
&=h^\varepsilon u^\varepsilon e^u_x+h^\varepsilon u^0_x e^u+u^\varepsilon u^0_x e^h-u^0_x e^he^u \ ,
\end{aligned}
$$
which yields 
\begin{equation}\label{eq:hineq2}
\Omega_2=(h^\varepsilon u^\varepsilon u^\varepsilon_x-h^0 u^0 u^0_x , e^u) \lesssim \|e^h\|^2+\|e^u\|^2+\|e^u_x\|^2 \ .
\end{equation}

To estimate the quantity $\gamma^\varepsilon-\gamma^0$ we estimate the terms of this difference separately:
First, we write
$$
\left([h^\varepsilon]^3u^\varepsilon_{xt}-[h^0]^3u^0_{xt},e^u_x\right) = \left((h^\varepsilon)^3e^u_{xt},e^u_x \right)+\left(([h^\varepsilon]^3-[h^0]^3 )u^0_{xt},e^u_x\right)\ ,
$$
which implies
\begin{equation}\label{eq:compeq2}
\left([h^\varepsilon]^3u^\varepsilon_{xt}-[h^0]^3u^0_{xt},e^u_x\right) = 
\tfrac{1}{2}\frac{d}{dt}\int_{-\infty}^\infty [h^\varepsilon]^3 [e^u_x]^2~dx-\tfrac{3}{2}\left([h^\varepsilon]^2h^\varepsilon_t,[e^u_x]^2 \right) +\left( ([h^\varepsilon]^3-[h^0]^3)u^0_{xt},e^u_x \right)\ .
\end{equation}
% which is leading to
% \begin{equation}\label{eq:compeq2}
% \begin{aligned}
% \left([h^\varepsilon]^3u^\varepsilon_{xt}-[h^0]^3u^0_{xt},e^u_x\right) &= 
% \tfrac{1}{2}\frac{d}{dt}\int_{-\infty}^\infty [h^\varepsilon]^3 [e^u_x]^2~dx-\tfrac{3}{2}\left([h^\varepsilon]^2h^\varepsilon_t,[e^u_x]^2 \right) +\left( ([h^\varepsilon]^3-[h^0]^3)u^\varepsilon_{xt},e^u_x \right)\\
% & \quad
% -\tfrac{1}{2}\frac{d}{dt}\int_{-\infty}^\infty ([h^\varepsilon]^3-[h^0]^3)[e^u_x]^2~dx +\tfrac{3}{2}\left([h^\varepsilon]^2h^\varepsilon_t-[h^0]^2h^0_t,[e^u_x]^2\right)\ .
% \end{aligned}
% \end{equation}
The next term in the expression $\gamma^\varepsilon-\gamma^0$ can be written
$$
\begin{aligned}
\left[h^\varepsilon\right]^3 u^\varepsilon u^\varepsilon_{xx}-[h^0]^3u^0 u^0_{xx} &= [h^\varepsilon]^3u^\varepsilon u^\varepsilon_{xx}-[h^0]^3u^0(u^\varepsilon_{xx}-e^u_{xx})\\
&=[h^0]^3u^0e^u_{xx}+([h^\varepsilon]^3 u^\varepsilon-[h^0]^3u^0)u^\varepsilon_{xx}\ .
\end{aligned}
$$
Since
$$
[h^\varepsilon]^3 u^\varepsilon-[h^0]^3u^0=[h^\varepsilon]^3 u^\varepsilon-[h^0]^3(u^\varepsilon- e^u)=([h^\varepsilon]^3-[h^0]^3)u^\varepsilon+[h^0]^3 e^u\ ,
$$
we have that
$$
[h^\varepsilon]^3 u^\varepsilon u^\varepsilon_{xx}-[h^0]^3u^0 u^0_{xx}=[h^0]^3u^0 e^u_{xx}+([h^\varepsilon]^3-[h^0]^3)u^\varepsilon u^\varepsilon_{xx}+[h^0]^3 u^\varepsilon_{xx}e^u\ .
$$
Therefore,
$$
\begin{aligned}
([h^\varepsilon]^3 u^\varepsilon u^\varepsilon_{xx}-[h^0]^3u^0 u^0_{xx},e^u_x) &=-\tfrac{1}{2}\left(3[h^0]^2h^0_xu^0+[h^0]^3u^0_x,[e^u_x]^2\right)+\left(([h^\varepsilon]^3-[h^0]^3)u^0 u^0_{xx},e^u_x \right) \\
&+([h^0]^3u^\varepsilon_{xx}e^u,e^u_x)\ ,
\end{aligned}
$$
which leads to the inequality
\begin{equation}\label{eq:hineq3}
\Omega_3=\tfrac{\delta}{3}([h^\varepsilon]^3 u^\varepsilon u^\varepsilon_{xx}-[h^0]^3u^0 u^0_{xx},e^u_x) \lesssim \|e^h\|^2+\|e^u\|^2+\|e^u_x\|^2 \ .
\end{equation}

The last term in the expression $\gamma^\varepsilon-\gamma^0$ can be written
$$
\begin{aligned}
 \left[h^\varepsilon\right]^3[u^\varepsilon_x]^2-[h^0]^3[u^0_x]^2 &=[h^\varepsilon]^3 [u^\varepsilon_x]^2 -[h^0]^3(u^\varepsilon_x-e^u_x)^2\\
&=([h^\varepsilon]^3-[h^0]^3)[u^\varepsilon_x]^2+2[h^0]^3u^\varepsilon_xe^u_x-[h^0]^3[e^u_x]^2\\
&=([h^\varepsilon]^3-[h^0]^3)[u^\varepsilon_x]^2+[h^0]^3u^\varepsilon_xe^u_x+[h^0]^3u^0_xe^u_x\ .
\end{aligned}
$$
Thus,
\begin{equation}\label{eq:hineq4}
\Omega_4=\tfrac{\delta}{3}([h^\varepsilon]^3[u^\varepsilon_x]^2-[h^0]^3[u^0_x]^2,e^u_x) \lesssim \|e^h\|^2+\|e^u_x\|^2\ .
\end{equation}

Finally, the first term on the left-hand side of (\ref{eq:compeq1}) can be written as
$$
h^\varepsilon u^\varepsilon_t-h^0 u^0_t = h^\varepsilon e^u_t-e^he^u_t+u^\varepsilon_te^h\ ,
$$
and this yields
\begin{equation}
\begin{aligned}\label{eq:compeq3}
(h^\varepsilon u^\varepsilon_t-h^0 u^0_t,e^u)&= (h^\varepsilon e^u_t,e^u)-(e^he^u_t,e^u)+(u^\varepsilon_t e^h,e^u)\\
&=\tfrac{1}{2}\frac{d}{dt}\int_{-\infty}^\infty h^\varepsilon [e^u]^2~dx -\tfrac{1}{2}(h^\varepsilon_t,[e^u]^2)-\tfrac{1}{2}\frac{d}{dt}\int_{-\infty}^\infty e^h[e^u]^2~dx\\
&\quad +\tfrac{1}{2}(e^h_t,[e^u]^2)+(u^\varepsilon_t e^h,e^u)\ .
\end{aligned}
\end{equation}

Adding the two equations (\ref{eq:firsten}), (\ref{eq:compeq1})  and substituting (\ref{eq:compeq2}), (\ref{eq:compeq3})  we obtain
$$
\begin{aligned}
\tfrac{1}{2}\frac{d}{dt}\|e^h\|^2 &+\tfrac{1}{2}\frac{d}{dt}\int_{-\infty}^\infty \left(h^\varepsilon [e^u]^2+\tfrac{\delta}{3}[h^\varepsilon]^3 [e^u_x]^2 \right)~dx =\tfrac{1}{2}\frac{d}{dt}\int_{-\infty}^\infty e^h[e^u]^2~dx +\tfrac{1}{2}(h^\varepsilon_t,[e^u]^2)\\
&\quad -\tfrac{1}{2}(e^h_t,[e^u]^2)-(u^\varepsilon_t e^h,e^u)
 -\left((h^\varepsilon e^u)_x,e^h\right) -\tfrac{1}{2}\left(u^\varepsilon_xe^h,e^h \right) +\tfrac{1}{2}\left((e^h)^2, e^u_x \right)\\
&\quad +\tfrac{\delta}{2}\left([h^\varepsilon]^2h^\varepsilon_t,[e^u_x]^2 \right) -\tfrac{\delta}{3}\left( ([h^\varepsilon]^3-[h^0]^3)u^0_{xt},e^u_x \right)
 -\varepsilon(q^\varepsilon_x,e^u_x)+\Omega_1-\Omega_2-\Omega_3-\Omega_4\ .
\end{aligned}
$$
Using the inequalities (\ref{eq:hineq1}), (\ref{eq:hineq2}) and (\ref{eq:hineq3}) we get
$$
\tfrac{1}{2}\frac{d}{dt}\|e^h\|^2+\tfrac{1}{2}\frac{d}{dt}\int_{-\infty}^\infty \left(h^\varepsilon [e^u]^2+\tfrac{\delta}{3}[h^\varepsilon]^3 [e^u_x]^2 \right)~dx
\lesssim \varepsilon\|e^u_x\|+ \|e^h\|^2+\|e^u\|^2+\|e^u_x\|^2\ .
$$
Integrating with respect to $t$, and using Lemma \ref{lem:lemmab1} and the assumption (\ref{eq:cavity}) where $h^\varepsilon\geq c_0>0$ for all $\varepsilon\in(0,1)$ we obtain
\begin{equation}
\mathcal{E}^2(t) \lesssim \|e^h\|^2+\int_{-\infty}^\infty \left(h^\varepsilon [e^u]^2+\tfrac{\delta}{3}[h^\varepsilon]^3 [e^u_x]^2 \right)~dx \lesssim  \int_0^t  \varepsilon \mathcal{E}(\tau)  + \varepsilon^{-1} \mathcal{E}^3(\tau)~d\tau\ .
\end{equation}
Therefore, from Lemma \ref{lem:lemmab2} we have that
$$\mathcal{E}(t) \lesssim \varepsilon t\ ,$$
which completes the proof of the theorem.
\end{proof}

\begin{figure}[ht!]
  \centering
\includegraphics[width=\columnwidth]{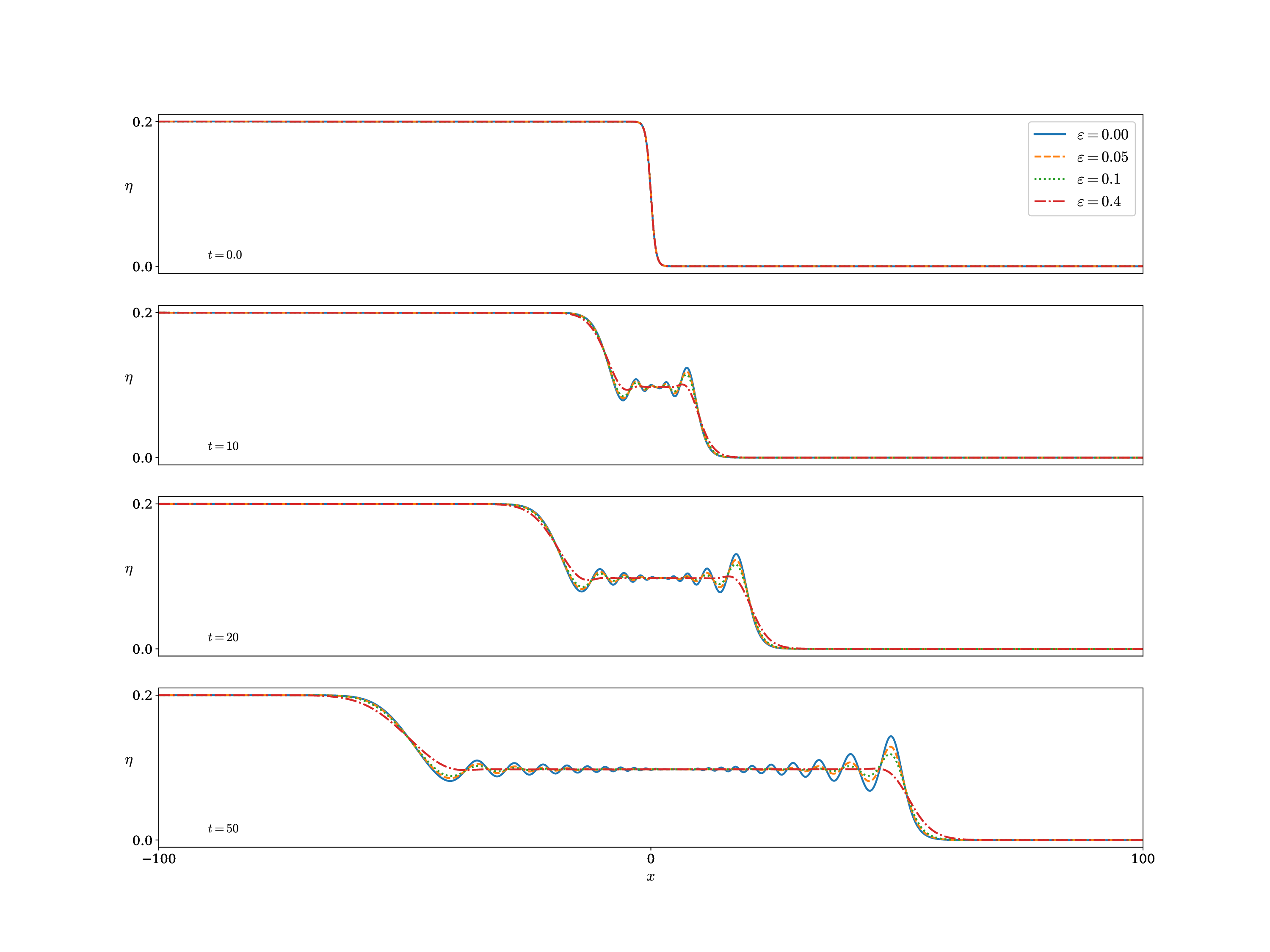}
  \caption{Comparison between of oscillatory shock waves of the modified SGN system and a dispersive shock wave of the classical SGN system ($\delta=1/3$)}
  \label{fig:comparison}
\end{figure}

Experimentally, we generate undular bores by allowing a heap of water to collapse under gravity almost instantaneously. This is known as the dam-break problem, imagining a hypothetical dam disappearing instantaneously. Such experiments can be simulated numerically using initial data that resembles the regularized Riemann data shown in Figure \ref{fig:comparison}. Figure \ref{fig:comparison} illustrates the generation of waves for a dam-break problem with a dispersion coefficient of $\delta=1/3$. Specifically, we observe the difference between oscillatory shock waves generated for different values of $\varepsilon>0$ and a dispersive shock wave of the SGN equations for $\varepsilon=0$. The top figure shows the initial condition used for $h(x,0)$ which originated from regularized Riemann data. The initial velocity profile was $u(x,0)=0$. In this particular experiment and for the times we present all the shock waves travel at very similar speeds and the solutions remain relatively close. For the numerical solution of the Serre equations and its modification we used the standard Galerkin/Finite element method with periodic boundary conditions analysed and justified in \cite{MID2014,ADM2017}.

\subsection{Convergence to entropic shocks}

The convergence of the dissipative-dispersive shock waves of the KdV-Burgers equations as $\delta$ and/or $\varepsilon$ approach zero has been studied in \cite{BS1985,PR2007}. A more generalized analysis, which apparently includes the solutions we studied in this work has been presented in \cite{BMT2023}.

In particular, we focus on equation (\ref{ODE2}), expressed as:
\begin{equation}
z'' + cz' + \Phi'(z) = 0\ ,
\end{equation}
where $\Phi'(z) = \frac{3}{K_1^2}F(e^z)$ represents the derivative of the potential function given in equation (\ref{eq:potential}). As demonstrated in Section \ref{sec:existence}, we observe that $\Phi'(z) < 0$ for $z_l < z < z_r$, $\Phi'(z) > 0$ for $z_r < z < \bar{z}$, and we also find that $\Phi''(z_l) < 0$ and $\Phi''(z_r) > 0$. Furthermore, in the same section, we established that $0 = \Phi(z_r) < \Phi(z) < \Phi(z_l) = \Phi(\bar{z})$ for $z$ within the interval $(z_l, \bar{z})$. According to \cite[Prop. 6]{BMT2023}, the solution can be divided into two distinct regions:
\begin{itemize}
\item A region characterized by large-amplitude oscillations, with a size of $O(1/c)$.
\item A region characterized by small-amplitude oscillations, also with a length of $O(1/c)$.
\end{itemize}
These oscillations, in terms of the original variables, reduce to sizes of $O(\sqrt{\delta}/c) = O(\delta/\varepsilon)$ and tend towards zero as $\delta$ and $\varepsilon$ approach zero, provided that $\delta = o(\varepsilon)$. In this scenario, the dissipative-dispersive shocks of the modified SGN equations converge to the classical shock waves of the non-dispersive shallow-water wave equations.

\section{Conclusions}

The classical non-dissipative Serre-Green-Nagdi system describes strongly nonlinear and weakly dispersive water waves, making it an attractive choice for studying undular bores. However, experimental observations suggest that to accurately approximate undular bores occurring in natural settings, some level of dissipation due to turbulence is necessary. Unfortunately, there is no existing theory that addresses dissipative strongly nonlinear waves. For this reason, we have introduced a modification to the classical Serre-Green-Naghdi system by incorporating an artificial term. We have justified the use of this specific term through several observations:
\begin{enumerate}
    
\item[(i)] We have demonstrated that this modified system supports oscillatory and regularized traveling wave solutions that closely resemble natural undular bores.

\item[(ii)] The solutions of the modified system converge to the non-dissipative solutions as dissipation approaches zero explaining the shape of the dispersive shock waves at initial stages of their formation.

\item[(iii)] In cases where both dispersion and dissipation diminish at a certain rate, the dissipative-dispersive shocks tend to evolve into entropic shock waves.
\end{enumerate}

Our findings are supported further by comparisons with laboratory data, providing additional empirical validation for our theoretical evidence.

A summary of the characteristic values of dissipative-dispersive shock waves of the modified Serre-Green-Nagdi equations are presented in Table \ref{tab:summary}.

\begin{table}[ht!]
\begin{tabular}{lc}
Phase speed & $s=(q_l-q_r)/(h_l-h_r)$ \\\hline
Right-side limits & $h_r$, $q_r$ given \\\hline
Left-side  depth limit & $h_l=(-h_r + \sqrt{h_r^2 + 8h_r( q_r/h_r- s)^2})/2$\\\hline
Left-side momentum limit & $q_l= (2 q_r - 3 s h_r  +s \sqrt{h_r^2 + 8 h_r (q_r/h_r -  s)^2})/2$ \\\hline
Oscillatory shock wave & $s^2\varepsilon^2/\delta<4(s h_l-q_l)^2(h_l^3-K_1^2)/3h_l^2$ \\\hline
Regularized shock wave & $s^2\varepsilon^2/\delta\geq 4(s h_l-q_l)^2(h_l^3-K_1^2)/3h_l^2$
\end{tabular}
\caption{Summary of characteristic values}\label{tab:summary}
\end{table}

%\bibliographystyle{plain} % We choose the "plain" reference style
%\bibliography{biblio}

\end{document}